\documentclass[12pt]{amsart}
\pdfoutput=1
\providecommand{\lang}{UKenglish}
\providecommand{\KOMAopt}{twoside=semi, DIV=default, headinclude}

\usepackage{eurosym}
\usepackage[\lang]{babel}
\usepackage[dvipsnames,svgnames]{xcolor}
\definecolor{darkblue}{rgb}{0.0,0.0,0.6}
\usepackage[utf8]{inputenc}
\usepackage[T1]{fontenc}
\usepackage{ae}
\usepackage{aecompl}
\usepackage{lmodern}
\usepackage{csquotes}
\usepackage[babel=true]{microtype}
\usepackage{scrextend}
\usepackage[\KOMAopt]{typearea}
\usepackage{amsmath}
\usepackage{mathtools}
\usepackage{centernot}
\usepackage{amssymb}
\usepackage{amsfonts}
\usepackage{amsthm}
\usepackage{mathrsfs}
\usepackage{booktabs}
\usepackage{pdfpages}
\usepackage{bm}
\usepackage[only,llbracket,rrbracket]{stmaryrd}
\usepackage{upgreek}

\makeatletter
\@namedef{subjclassname@2020}{\textup{2020} Mathematics Subject Classification}
\makeatother

\usepackage[pdflang=en-UK, colorlinks, urlcolor=darkblue, linkcolor=darkblue, citecolor=darkblue]{hyperref}
\usepackage{color}
\usepackage{tikz}
	\usetikzlibrary{calc}
	\usetikzlibrary{arrows,decorations.markings}
\usepackage{tikz-cd}

\usepackage{todonotes}

\usepackage{enumitem}
\setenumerate{label=(\alph*)}

\usepackage{subcaption}

\newcommand{\op}{\mathrm{op}}
\newcommand{\proj}[1]{P_{#1}}

\newcommand{\donothing}[1]{}

\newcommand{\Ext}{\operatorname{Ext}}
\newcommand{\Hom}{\operatorname{Hom}}

\newcommand{\im}{\operatorname{im}}
\newcommand{\id}{\operatorname{id}}

\newcommand{\iso}{\cong}
\newcommand{\isoto}{\stackrel{\sim}{\to}}

\newcommand{\Ob}{\operatorname{Ob}}

\newcommand{\SL}{\mathrm{SL}}

\newcommand{\cohom}[1]{\mathrm{H}^{#1}}

\newcommand{\bdd}{\mathrm{b}}
\newcommand{\CC}{\mathbb{C}}

\newcommand{\KK}{\mathbb{K}}

\newcommand{\QQ}{\mathbb{Q}}

\newcommand{\ZZ}{\mathbb{Z}}

\newcommand{\defeq}{\coloneqq}

\newcommand{\clustalg}[1]{\mathscr{A}_{#1}}

	\newcommand{\add}{\operatorname{add}}
\newcommand{\clustcat}[1]{\mathcal{C}_{#1}}

	\newcommand{\bdcat}{\mathcal{D}^{\bdd}}

	\newcommand{\bhcat}{\mathcal{K}^{\bdd}}
\newcommand{\indec}{\operatorname{indec}}

	\newcommand{\fgmod}{\operatorname{mod}}

\newcommand{\projcat}{\operatorname{proj}}

\newcommand{\rep}{\operatorname{rep}}

\renewcommand{\epsilon}{\varepsilon}

  \theoremstyle{plain}
\newtheorem{thm}{Theorem}[section]
\newtheorem{thm*}{Theorem}

\newtheorem{cor*}{Corollary}
\newtheorem{prop}[thm]{Proposition}

\newtheorem{phen}{Phenomenon}

\theoremstyle{definition}
\newtheorem{defn}[thm]{Definition}
\newtheorem{eg}[thm]{Example}

\theoremstyle{remark}
\newtheorem{rem}[thm]{Remark}

\numberwithin{equation}{section} \usepackage[backend=biber,citestyle=numeric-comp,bibstyle=numeric,maxbibnames=99,giveninits=true,doi=false,isbn=false,url=false,eprint=true]{biblatex}

\renewbibmacro{in:}{\ifentrytype{article}{}{\printtext{\bibstring{in}\intitlepunct}}}

\DeclareFieldFormat[article,inbook,incollection,inproceedings,patent]{title}{#1}
\DeclareFieldFormat[thesis,unpublished]{title}{{\it #1}}

\DeclareFieldFormat[online]{date}{Preprint (#1)}

\DeclareFieldFormat[article]{volume}{\mkbibbold{#1}}
\renewbibmacro*{volume+number+eid}{\printfield{volume}\setunit{\addcomma\space}\printfield{eid}}

\DeclareFieldFormat[book,incollection]{number}{Vol.~#1}
\renewbibmacro*{series+number}{\iffieldundef{series}{}
    {\printfield{series}\iffieldundef{number}{}{\setunit*{\addcomma\addspace}\printfield{number}\newunit}}}

\renewbibmacro*{publisher+location+date}{\printlist{publisher}\setunit*{\addcomma\space}\usebibmacro{date}\newunit}

\DeclareFieldFormat{eprint:arxiv}{\ifentrytype{online}
  {\ifhyperref
    {\href{http://arxiv.org/abs/#1}{\nolinkurl{arXiv:#1}}}
    {\nolinkurl{arXiv:#1}}
   \iffieldundef{eprintclass}
    {}
    {{\texttt{[\thefield{eprintclass}]}}}}
  {\iffieldundef{eprintclass}
    {\mkbibparens{\ifhyperref
    {\href{http://arxiv.org/abs/#1}{\nolinkurl{arXiv:#1}}}
    {\nolinkurl{arXiv:#1}}}}
    {\mkbibparens{\ifhyperref
    {\href{http://arxiv.org/abs/#1}{\nolinkurl{arXiv:#1}}}
    {\nolinkurl{arXiv:#1}}
    {\texttt{[\thefield{eprintclass}]}}}}}}

\usepackage{xpatch}

\makeatletter
\patchcmd\blx@bblinput{\blx@blxinit}
                      {\blx@blxinit
                      }{}{\fail}
\makeatother

\title{From frieze patterns to cluster categories}
\author{Matthew Pressland}
\address[Matthew Pressland]{School of Mathematics \& Statistics\\University of Glasgow\\University Place\\Glasgow G20 8LR\\United Kingdom}
\email{\href{mailto:Matthew.Pressland@glasgow.ac.uk}{Matthew.Pressland@glasgow.ac.uk}}
\urladdr{\url{http://mdpressland.github.io}}

\renewcommand{\indec}{\operatorname{ind}}
\renewcommand{\bdcat}[1]{\mathcal{D}^{\bdd}(#1)}
\renewcommand{\bhcat}[1]{\mathcal{K}^{\bdd}(#1)}
\renewcommand{\cohom}[2]{\mathrm{H}^{#1}(#2)}
\renewcommand{\proj}{\projcat}

\begin{document}

\begin{abstract}
Motivated by Conway and Coxeter's combinatorial results concerning frieze patterns, we sketch an introduction to the theory of cluster algebras and cluster categories for acyclic quivers. The goal is to show how these more abstract theories provide a conceptual explanation for phenomena concerning friezes, principally integrality and periodicity.
\end{abstract}
\maketitle

\section{Introduction}

Since their introduction by Fomin and Zelevinsky around the turn of the millennium \cite{FomZel-CA1}, the theory of cluster algebras has had a significant influence on a wide variety of other areas of mathematics. Originally introduced to study positivity and canonical bases in Lie theory \cite{geissclusteralgebra,GHKK,cheunggreedy,Qin-Bases}, applications have been found in combinatorics \cite{BMDMTY,IngTho}, Teichmüller theory \cite{FocGon,FomZel-CA2}, algebraic geometry and mirror symmetry \cite{Bridgeland-Scattering,Wemyss-HomMMP,grossbirational,augusttilting}, integrable systems \cite{Keller-Periodicity,fordydiscrete}, mathematical physics \cite{BKM,Franco-BFTs} and beyond---the references given here are by no means exhaustive, and more can be found in the introduction to \cite{kellercluster}. In these notes we will focus on the influence of cluster algebras on the representation theory of finite-dimensional algebras, leading to the development of cluster categories by Caldero, Chapoton and Schiffler \cite{calderoquivers1,calderoquivers2} and Buan, Marsh, Reineke, Reiten and Todorov \cite{BMRRT} (and later, in more generality than covered here, by Amiot \cite{Amiot-ClustCat}).

We start in Section~\ref{s:friezes} with the combinatorics of frieze patterns, first studied by Coxeter \cite{coxeterfrieze} and then Conway--Coxeter \cite{conwaytriangulated1,conwaytriangulated2} several decades before the introduction of cluster algebras. These patterns consist of an infinite (but periodic) strip of positive integers, and exhibit many remarkable combinatorial phenomena. To shed light on these phenomena, we will explain in Section~\ref{s:clust-alg} how the entries of the frieze can be viewed as specialisations of cluster variables, certain Laurent polynomials which are distinguished elements of a cluster algebra. Continuing this approach, further combinatorial properties of friezes, including the aforementioned periodicity, will be explained in Section~\ref{s:clust-cat} via categorification, whereby a representation-theoretically defined cluster category is introduced to provide a new perspective on the combinatorics of cluster algebras.

In this way we aim to illustrate how each successively more abstract framework---first passing from friezes to cluster algebras, and then to cluster categories---can lead to clean explanations of phenomena appearing at the previous level. This path can be followed further than we have space to do here, and leads to, for example, the theory of quivers with potential \cite{DWZ1} and their more general cluster categories \cite{Amiot-ClustCat}, Frobenius cluster categories and their applications to cluster algebras appearing in geometry \cite{BIRSc, GLS-PFVs, JKS, Pressland-iCY, Pressland-Postnikov}, Adachi--Iyama--Reiten's $\tau$-tilting theory \cite{adachitau}, and the theory of stability conditions and scattering diagrams \cite{Bridgeland-Scattering,BST,GHKK}.

These notes are not intended as a comprehensive exposition of the subjects we cover, but rather as a concise introduction which we hope the curious reader will use as a jumping-off point for further study. To that end, we give here a short list of suggestions for further reading covering these topics in more detail, more of which can be found in the references throughout the notes. To the reader looking for a more detailed survey of the representation-theoretic approach to cluster algebras, we recommend Keller \cite{kellercluster}. Several books on cluster algebras are available: that by Gekhtman, Shapiro and Vainshtein \cite{gekhtmancluster} focusses on connections to Poisson geometry, whereas that by Marsh \cite{marshlecture} covers many combinatorial aspects of the theory. A further book, by Fomin, Williams and Zelevinsky, is in progress, and (at the time of writing) the first seven chapters can be found on arXiv \cite{FWZ1-3,FWZ4-5,FWZ6,FWZ7}. Up-to-date information concerning the development of the theory of cluster algebras and its many applications can be found on the Cluster Algebras Portal \cite{clusterportal}.

Schiffler's recent book on representations of quivers \cite{schifflerquiver} contains a section on the properties of cluster-tilted algebras, a class of algebras we will mention briefly in Section~\ref{s:clust-cat} and which arises from cluster categories. Also recommended are books on representations of quivers and other finite-dimensional algebras by Assem--Simson--Skowroński \cite{assemelements} and Auslander--Reiten--Smalø \cite{auslanderrepresentationtheory}. 

For a more substantial discussion of the rich combinatorial theory of friezes, we recommend the survey by Morier-Genoud \cite{moriergenoudcoxeter} and a paper by Propp \cite{proppcombinatorics}.

\section{Frieze Patterns}

Frieze patterns were first introduced by Coxeter \cite{coxeterfrieze} in 1971. Many of their combinatorial properties can be found in a paper by Conway and Coxeter published in two parts, the first \cite{conwaytriangulated1} consisting of a list of problems, the solutions to which are provided in the second \cite{conwaytriangulated2}.

\label{s:friezes}
\begin{defn}
\label{d:frieze}
A \emph{frieze pattern} (or simply a frieze) of \emph{height} $n$ consists of $n+2$ rows of positive integers, typically written in a slightly offset fashion as in the following example (with $n=6$): 
\begin{center}
\resizebox{0.95\textwidth}{!}{$\begin{tikzpicture}[ampersand replacement=\&]
\matrix[matrix of math nodes](m){
\cdots\&\phantom{14}\&\mathbf{1}\&\phantom{14}\&\mathbf{1}\&\phantom{14}\&\mathbf{1}\&\phantom{14}\&\mathbf{1}\&\phantom{14}\&\mathbf{1}\&\phantom{14}\&\mathbf{1}\&\phantom{14}\&\mathbf{1}\&\phantom{14}\&\mathbf{1}\&\phantom{14}\&\mathbf{1}\&\phantom{14}\&\mathbf{1}\&\phantom{14}\&\cdots\\
\&{3}\&\&{1}\& \&{2}\& \&{3}\& \&{2}\& \&{2}\& \&{2}\&\&1\&\&5\& \&3\& \&1\&\\
\cdots\& \&{2}\& \&{1}\& \&{5}\& \&{5}\& \&{3}\& \&{3}\& \&1\& \&4\& \&14\& \&2\& \&\cdots\\
\&9\& \&{1}\& \&{2}\& \&{8}\& \&{7}\& \&{4}\& \&1\& \&3\& \&11\& \&9\& \&1\&\\
\cdots\& \&4\& \&{1}\& \&{3}\& \&{11}\& \&{9}\& \&1\& \&2\& \&8\& \&7\& \&4\& \&\cdots\\
\&3\& \&3\& \&{1}\& \&{4}\& \&{14}\& \&2\& \&1\& \&5\& \&5\& \&3\& \&3\&\\
\cdots\& \&2\& \&2\& \&{1}\& \&{5}\& \&3\& \&1\& \&2\& \&3\& \&2\& \&2\& \&\cdots\\
\&\mathbf{1}\&\phantom{14}\&\mathbf{1}\&\phantom{14}\&\mathbf{1}\&\phantom{14}\&\mathbf{1}\&\phantom{14}\&\mathbf{1}\&\phantom{14}\&\mathbf{1}\&\phantom{14}\&\mathbf{1}\&\phantom{14}\&\mathbf{1}\&\phantom{14}\&\mathbf{1}\&\phantom{14}\&\mathbf{1}\&\phantom{14}\&\mathbf{1}\&\\
};
\end{tikzpicture}$}
\end{center}
The defining properties are that
\begin{enumerate}
\item every entry in the first and final row is $1$, and
\item the entries satisfy the \emph{$\SL_2$ diamond rule}, meaning that every local configuration
\[\begin{array}{ccc}
&b&\\
a& &d\\
&c&
\end{array}\]
satisfies $ad-bc=1$. In other words, the matrix $\left(\begin{smallmatrix}a&b\\c&d\end{smallmatrix}\right)$ is an element of $\SL_2(\ZZ)$.
\end{enumerate}
We call the first and last row of the frieze, consisting only of $1$s, \emph{trivial} rows, and will always write these entries in bold to emphasise their fixed value. Note that the height measures the number of non-trivial rows.
\end{defn}

Some readers may find Definition~\ref{d:frieze} a little informal, preferring to think of a frieze pattern as a function assigning integer values to some set giving the points of the lattice diagram on which these integers are drawn. We will take this viewpoint later on Section~\ref{s:clust-cat}, but for the moment remain circumspect about which set this indeed is.

Because of the $\SL_2$ diamond rule, any three of the frieze entries in a diamond configuration determine the fourth. As a consequence, there are several ways to specify a subset of entries of the frieze in such a way that the rest are determined inductively by this rule.

In these notes, we will be most interested in specifying a frieze by choosing a \emph{lightning bolt}, consisting of one entry in each row such that the chosen entries in each pair of adjacent rows form part of the same diamond. An example is given by the circled entries in the frieze pattern below.
\begin{center}
\resizebox{0.95\textwidth}{!}{$\begin{tikzpicture}[ampersand replacement=\&]
\matrix[matrix of math nodes](m){
\cdots\&\phantom{14}\&\mathbf{1}\&\phantom{14}\&\mathbf{1}\&\phantom{14}\&\mathbf{1}\&\phantom{14}\&\mathbf{1}\&\phantom{14}\&\mathbf{1}\&\phantom{14}\&\mathbf{1}\&\phantom{14}\&\mathbf{1}\&\phantom{14}\&\mathbf{1}\&\phantom{14}\&\mathbf{1}\&\phantom{14}\&\mathbf{1}\&\phantom{14}\&\cdots\\
\&{3}\&\phantom{1}\&{1}\& \&{2}\& \&{3}\& \&{2}\& \&{2}\& \&{2}\&\&1\&\&5\& \&3\& \&1\&\\
\cdots\& \&{2}\&\phantom{1}\&{1}\&\phantom{1}\&{5}\& \&{5}\& \&{3}\& \&{3}\& \&1\& \&4\& \&14\& \&2\& \&\cdots\\
\&9\&\phantom{1}\&{1}\&\phantom{1}\&{2}\& \&{8}\& \&{7}\& \&{4}\& \&1\& \&3\& \&11\& \&9\& \&1\&\\
\cdots\& \&4\& \&{1}\& \&{3}\& \&{11}\& \&{9}\& \&1\& \&2\& \&8\& \&7\& \&4\& \&\cdots\\
\&3\& \&3\& \&{1}\& \&{4}\& \&{14}\& \&2\& \&1\& \&5\& \&5\& \&3\& \&3\&\\
\cdots\& \&2\& \&2\& \&{1}\&\phantom{1}\&{5}\& \&3\& \&1\& \&2\& \&3\& \&2\& \&2\& \&\cdots\\
\&\mathbf{1}\&\phantom{14}\&\mathbf{1}\&\phantom{14}\&\mathbf{1}\&\phantom{14}\&\mathbf{1}\&\phantom{14}\&\mathbf{1}\&\phantom{14}\&\mathbf{1}\&\phantom{14}\&\mathbf{1}\&\phantom{14}\&\mathbf{1}\&\phantom{14}\&\mathbf{1}\&\phantom{14}\&\mathbf{1}\&\phantom{14}\&\mathbf{1}\&\\
};
\draw[rounded corners,thick]
(m-2-3.east) -- (m-2-4.north) -- (m-3-6.west) -- (m-4-5.west) -- (m-7-8.west) -- (m-8-7.north) -- (m-4-3.east) -- (m-3-4.east) -- cycle;
\end{tikzpicture}$}
\end{center}
Having chosen positive integers to fill the entries in a lightning bolt, we can compute the other entries of the frieze by repeated use of the diamond rule. (The figures in the proof of Theorem~\ref{t:cluster-frieze} below may be helpful in seeing this.) However, this process requires division---for example, if we have entries $a$, $b$ and $c$ in a diamond, then
\[d=\frac{1+bc}{a}.\]
Thus, even if $a$, $b$ and $c$ are all positive integers, as required by the definition of a frieze, there is no guarantee that $d$ will also be an integer (although it will at least be positive). The first phenomenon that we will be concerned with in these notes is the following.

\begin{phen}
\label{ph:lightning-bolt}
Given a lightning bolt, setting its entries equal to $1$ determines a unique frieze pattern. More concretely, we can compute all other entries via the diamond rule, and all of them are positive integers, as the definition requires.
\end{phen}

\begin{eg}
\label{eg:no-bolt}
Given Phenomenon~\ref{ph:lightning-bolt}, one might reasonably ask whether all frieze patterns are obtained by choosing a lightning bolt and settings its entries to $1$. The answer, however, is no; the smallest example exhibiting this, for which $n=3$, is shown below.

\[
\begin{tikzpicture}
\matrix[matrix of math nodes](m){
&\mathbf{1}& &\mathbf{1}& &\mathbf{1}& &\mathbf{1}& &\mathbf{1}& &\mathbf{1}& &\mathbf{1}& &\mathbf{1}&\\
\cdots& &1& &3& &1& &3& &1& &3& &1& &\cdots\\
&2& &2& &2& &2& &2& &2& &2& &2&\\
\cdots& &3& &1& &3& &1& &3& &1& &3& &\cdots\\
&\mathbf{1}& &\mathbf{1}& &\mathbf{1}& &\mathbf{1}& &\mathbf{1}& &\mathbf{1}& &\mathbf{1}& &\mathbf{1}&\\
};
\end{tikzpicture}
\]
\end{eg}

The alert reader may have already noticed that both of the examples of frieze patterns given so far are highly symmetric; indeed, their rows are periodic with period (dividing) $n+3$. This turns out to be the case for all frieze patterns, and is a consequence of a slightly stronger periodicity.

\begin{phen}
\label{ph:glide}
Any frieze pattern is periodic under a glide reflection, for which a fundamental domain (excluding the trivial rows) consists of a triangle with $n+1$ elements in the base, but missing its peak, as shown in the example:
\begin{center}
\resizebox{0.95\textwidth}{!}{$\begin{tikzpicture}[ampersand replacement=\&]
\matrix[matrix of math nodes](m){
\cdots\&\phantom{14}\&\mathbf{1}\&\phantom{14}\&\mathbf{1}\&\phantom{14}\&\mathbf{1}\&\phantom{14}\&\mathbf{1}\&\phantom{14}\&\mathbf{1}\&\phantom{14}\&\mathbf{1}\&\phantom{14}\&\mathbf{1}\&\phantom{14}\&\mathbf{1}\&\phantom{14}\&\mathbf{1}\&\phantom{14}\&\mathbf{1}\&\phantom{14}\&\cdots\\
\&{3}\&\phantom{1}\&{1}\& \&{2}\& \&{3}\& \&{2}\& \&{2}\& \&{2}\&\&1\&\phantom{1}\&5\& \&3\& \&1\&\\
\cdots\& \&{2}\& \&{1}\& \&{5}\& \&{5}\& \&{3}\& \&{3}\& \&1\& \&4\& \&14\& \&2\& \&\cdots\\
\&9\& \&{1}\& \&{2}\& \&{8}\& \&{7}\& \&{4}\& \&1\& \&3\& \&11\& \&9\& \&1\&\\
\cdots\& \&4\& \&{1}\& \&{3}\& \&{11}\& \&{9}\& \&1\& \&2\& \&8\& \&7\& \&4\& \&\cdots\\
\&3\& \&3\& \&{1}\& \&{4}\& \&{14}\& \&2\& \&1\& \&5\& \&5\& \&3\& \&3\&\\
\cdots\& \&2\& \&2\& \&{1}\& \&{5}\& \&3\& \&1\& \&2\& \&3\& \&2\& \&2\& \&\cdots\\
\&\mathbf{1}\&\phantom{14}\&\mathbf{1}\&\phantom{14}\&\mathbf{1}\&\phantom{14}\&\mathbf{1}\&\phantom{14}\&\mathbf{1}\&\phantom{14}\&\mathbf{1}\&\phantom{14}\&\mathbf{1}\&\phantom{14}\&\mathbf{1}\&\phantom{14}\&\mathbf{1}\&\phantom{14}\&\mathbf{1}\&\phantom{14}\&\mathbf{1}\&\\
};
\draw[rounded corners,thick]
(m-2-3.east) -- (m-2-4.north) -- (m-2-16.north) -- (m-2-17.west) -- (m-7-11.south east) -- (m-7-9.south west) -- cycle;
\end{tikzpicture}$}
\end{center}
\end{phen}
If we do include the trivial rows, a fundamental domain is given simply by a triangle with base having $n+2$ elements, but the above description will be more recognisable later on in Section~\ref{s:clust-cat}.

We will explain these phenomena, including the appearance of the frieze pattern in Example~\ref{eg:no-bolt}, in the coming sections, via the much more general theory of cluster algebras and cluster categories. However, before moving on to these topics, we take a digression to describe another way of generating frieze patterns, and a striking result of Conway and Coxeter.

\begin{defn}
Given a frieze pattern with height $n$, the (doubly-infinite) sequence of integers appearing in the first non-trivial row is called its \emph{quiddity sequence}. Since we have already remarked that this sequence is periodic of period $n+3$, we may equivalently take it to be a finite sequence of length $n+3$, considered up to cyclic reordering, i.e., up to the smallest equivalence relation with
\[(a_1,a_2,\dotsc,a_{n+3})\sim(a_{n+3},a_1,\cdots,a_{n+2}).\]
\end{defn}

Just as for the lightning bolt, knowing the quiddity sequence of a frieze pattern is sufficient to reconstruct the entire frieze, by repeated use of the diamond rule. Thus, given a sequence $(a_1,a_2,\dotsc,a_{n+3})$, we can draw the rows
\begin{center}
\resizebox{0.95\textwidth}{!}{$\begin{tikzpicture}[ampersand replacement=\&]
\matrix[matrix of math nodes](m){
\&\mathbf{1}\&\phantom{a_{n+3}}\&\mathbf{1}\&\phantom{a_{n+3}}\&\mathbf{1}\&\phantom{a_{n+3}}\&\mathbf{1}\& \&\cdots\&\phantom{a_{n+3}}\&\mathbf{1}\&\phantom{a_{n+3}}\&\mathbf{1}\&\phantom{a_{n+3}}\&\mathbf{1}\&\phantom{a_{n+3}}\&\mathbf{1}\&\\
\cdots\& \&a_{n+3}\& \&a_1\& \&a_2\& \&\cdots\& \&\cdots\& \&a_{n+3}\& \&a_1\& \&a_2\& \&\cdots\\
};
\end{tikzpicture}$}
\end{center}
and attempt to complete them to a frieze pattern below. However, now there are many obstructions to success---as well as the problem with the lightning bolt construction, whereby a priori one could obtain non-integer entries, there is no reason why the process of computing further rows of the frieze should not continue indefinitely, rather than terminating with a second trivial row of $1$s after the $n$ non-trivial rows we expect. Moreover, we now need to use the diamond rule in the situation in which $a$, $b$ and $d$ are known, and the computation
\[c=\frac{ad-1}{b},\]
involves subtraction. As a result, we also need to avoid the situation in which $a=d=1$ (when these are entries of a non-trivial row), because in that case we would compute $c=0$ in the next row, and thus fail to obtain a valid frieze.

Conway and Coxeter gave a remarkable method for producing all $(n+3)$-periodic sequences which are valid quiddity sequences, involving triangulations of polygons.

\begin{defn}
\label{d:triangulation}
Consider a convex polygon with $n+3$ sides, drawn in the plane. A \emph{triangulation} of this polygon is a maximal collection of pairwise non-crossing diagonals. In other words, it is a collection of diagonals which cut the polygon into triangles, as in the following example with $n=6$.
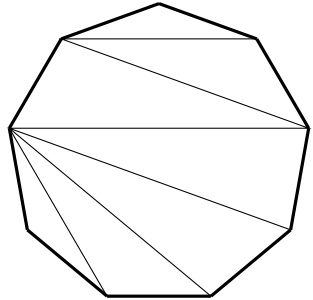
\[
\begin{tikzpicture}[scale=0.9]
\foreach \n in {0,1,2,3,4,5,6,7,8}
{\coordinate (v\n) at (90-40*\n:2);}
\foreach \n/\m in {0/1,1/2,2/3,3/4,4/5,5/6,6/7,7/8,8/0}
{\draw[very thick] (v\n) -- (v\m);}
\draw (v1) -- (v8) -- (v2) -- (v7) -- (v3) (v7) -- (v4) (v7) -- (v5);
\end{tikzpicture}
\]
Note that we do not consider the sides of the polygon to be diagonals---if we did, they would be contained in all triangulations.
\end{defn}

Each triangulation has an associated $(n+3)$-periodic sequence defined as follows. To each vertex $i$ of the polygon, attach the integer $a_i$ equal to the number of elements of $\mathcal{T}$ incident with $i$, plus $1$. (Visually, $a_i$ is the number of triangles incident with $i$ in the triangulation.) Since the vertices of the polygon are cyclically ordered (say clockwise), we obtain an $(n+3)$-periodic sequence $(a_1,\dotsc,a_{n+3})$. By using these sequences, Conway and Coxeter are able to classify frieze patterns in terms of triangulations.

\begin{thm}[Conway--Coxeter {\cite[(28)--(29)]{conwaytriangulated1,conwaytriangulated2}}]
\label{t:class1}
An $(n+3)$-periodic sequence is the quiddity sequence of a height $n$ frieze if and only if it arises from a triangulation of an $(n+3)$-gon as in the preceding paragraph. This gives a bijection between triangulations of polygons (up to rotation) and frieze patterns (up to translation).
\end{thm}

\begin{rem}
\label{r:class1}
Later on it will be more natural to consider frieze patterns up to the weaker symmetry of the glide reflection appearing in Phenomenon~\ref{ph:glide}. Modifying Theorem~\ref{t:class1}, these friezes are also classified by triangulations of an $(n+3)$-gon, but one in which the vertices are numbered clockwise by $1,\dotsc,n+3$ to break the symmetry of this polygon.
\end{rem}

For the example of a triangulation in Definition~\ref{d:triangulation}, we obtain the $(n+3)$-periodic sequence $(1,2,3,2,2,2,1,5,3)$, which is the quiddity sequence of the frieze in Definition~\ref{d:frieze}. The frieze pattern in Example~\ref{eg:no-bolt}, not arising from a lightning bolt, has quiddity sequence $(1,3,1,3,1,3)$, which corresponds to the triangulation
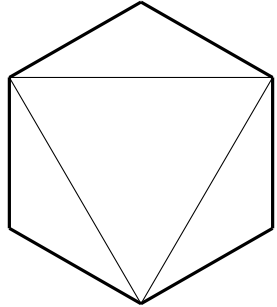
\[
\begin{tikzpicture}[scale=0.9]
\foreach \n in {0,1,2,3,4,5}
{\coordinate (v\n) at (90-60*\n:2);}
\foreach \n/\m in {0/1,1/2,2/3,3/4,4/5,5/0}
{\draw[very thick] (v\n) -- (v\m);}
\draw (v1) -- (v5) -- (v3) -- (v1);
\end{tikzpicture}
\]
of the hexagon. It is not a coincidence that the frieze not arising from a lightning bolt corresponds to a triangulation featuring an internal triangle, not involving any edges of the polygon, and the frieze that does arise from a lightning bolt corresponds to a triangulation with no internal triangles.

We close this section with an experiment. Returning to the lightning bolt method for constructing friezes, rather than setting the entries in our lightning bolt equal to $1$, or to other positive integers, we can fill them with formal variables $x_1,\dotsc,x_n$. Using the diamond rule, we can then construct a `frieze' whose entries are rational functions in $x_1,\dotsc,x_n$. We do this below for the simplest case of a height two frieze (for which any two lightning bolts are related by translation and reflection).
\begin{center}
\resizebox{0.95\textwidth}{!}{$\begin{tikzpicture}[ampersand replacement=\&]
\matrix[matrix of math nodes](m){
\&\mathbf{1}\&\phantom{1+x_2}\&\mathbf{1}\&\&\mathbf{1}\&\&\mathbf{1}\&\phantom{1+x_2}\&\mathbf{1}\&\\
\cdots\&\&x_1\&\&\dfrac{1+x_2}{x_1}\&\&\dfrac{1+x_1}{x_2}\&\&x_2\&\&\cdots\\
\&\dfrac{1+x_1}{x_2}\&\&x_2\&\&\dfrac{1+x_1+x_2}{x_1x_2}\&\&x_1\&\&\dfrac{1+x_2}{x_1}\&\\
\cdots\&\&\mathbf{1}\&\phantom{1+x_2}\&\mathbf{1}\&\&\mathbf{1}\&\phantom{1+x_2}\&\mathbf{1}\&\&\cdots\\
};
\end{tikzpicture}$}
\end{center}
This process exhibits  a \emph{Laurent phenomenon}; all of the rational functions appearing are Laurent polynomials, meaning they have a monomial denominator. To see that this is not a priori clear, observe that when computing the entry $\frac{1+x_1}{x_2}$ via the $\SL_2$ diamond rule from the three entries to its left, we calculate
\[
\dfrac{1+\dfrac{1+x_1+x_2}{x_1x_2}}{\dfrac{1+x_2}{x_1}}=\dfrac{x_1(1+x_1+x_2+x_1x_2)}{x_1x_2(1+x_2)}
=\dfrac{(1+x_1)(1+x_2)}{x_2(1+x_2)}
=\dfrac{1+x_1}{x_2},
\]
and see that the calculation initially produces a non-monomial denominator, but miraculously this has a common factor with the numerator, in such a way that we are left with a monomial again after simplifying. Note that the integrality of Phenomenon~\ref{ph:lightning-bolt} follows directly from this Laurent phenomenon---evaluating a Laurent polynomial in $x_1,\dotsc,x_n$ at the point $(1,\dotsc,1)$ will always give an integer.

The five Laurent polynomials appearing in the non-trivial rows of the above `frieze' turn out to be the five cluster variables of a cluster algebra of type $\mathsf{A}_2$. In the next section, we will explain what cluster algebras and their cluster variables are in some generality, and explain the connection between friezes of height $n$ and cluster algebras of type $\mathsf{A}_n$.

\section{Cluster algebras}
\label{s:clust-alg}

In this section, we introduce cluster algebras. These algebras were first defined by Fomin--Zelevinsky \cite{FomZel-CA1}, who (together with Berenstein) developed the theory in a series of seminal papers \cite{BFZ-CA3,FomZel-CA2,FomZel-CA4}. For our purposes, it will be sufficient to restrict to the case of cluster algebras arising from quivers, and without frozen variables.

A quiver is a directed graph. More precisely, it is a tuple $Q=(Q_0,Q_1,h,t)$, where $Q_0=\{1,\dotsc,n\}$ is a set of vertices, $Q_1$ is a set of arrows, and $h,t\colon Q_1\to Q_0$ are functions specifying the head and tail respectively of each arrow. We represent quivers visually, as in the examples below.
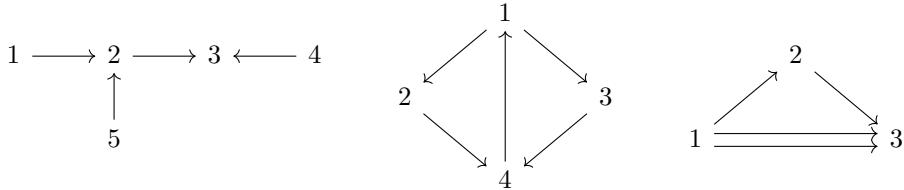
\[
\begin{tikzcd}[column sep=17pt,row sep=15pt]
1\arrow{r}&2\arrow{r}&3&4\arrow{l}\\&5\arrow{u}
\end{tikzcd}
\qquad
\begin{tikzcd}[column sep=17pt,row sep=15pt]
&1\arrow{dl}\arrow{dr}\\
2\arrow{dr}&&3\arrow{dl}\\
&4\arrow{uu}
\end{tikzcd}
\qquad
\begin{tikzcd}[column sep=17pt,row sep=15pt]
&2\arrow{dr}\\
1\arrow{ur}\arrow[shift left]{rr}\arrow[shift right]{rr}&&3
\end{tikzcd}
\]

\begin{defn}
A \emph{cluster quiver} is a quiver $Q$ without oriented cycles of length $1$ or $2$. In other words, no arrow $a\in Q_1$ can have $h(a)=t(a)$ (there are \emph{no loops}) and the configuration $\begin{tikzcd}i\arrow[shift left]{r}&j\arrow[shift left]{l}\end{tikzcd}$ is not permitted (there are \emph{no $2$-cycles}).
\end{defn}

\begin{rem}
Given a quiver $Q$ with $Q_0=\{1,\dotsc,n\}$, we can consider its \emph{signed adjacency matrix}, the $n\times n$ matrix $M(Q)$ with entries
\begin{equation}
\label{eq:MQ}
m_{ij}=\#\{\text{arrows $i\to j$ in $Q_1$}\} - \#\{\text{arrows $j\to i$ in $Q_1$}\}.
\end{equation}
Note that $M(Q)$ is skew-symmetric, meaning its transpose is equal to its negative. Fomin--Zelevinsky define cluster algebras in terms of such skew-symmetric matrices (or more generally, skew-symmetrisable matrices, defined as those which become skew-symmetric after multiplication with some diagonal matrix), whereas we opt for the more graphical language of quivers. The condition that $Q$ is a cluster quiver means that only one term on the right-hand side of \eqref{eq:MQ} may be non-zero, and that $Q$ can be reconstructed up to isomorphism from the data of $M(Q)$. Indeed, given a skew-symmetric $n\times n$ matrix $M$, the quiver $Q(M)$ with vertex set $\{1,\dotsc,n\}$ and with $\max\{m_{ij},0\}$ arrows from $i$ to $j$ is the unique cluster quiver with $M(Q(M))=M$.
\end{rem}

\begin{defn}
\label{d:quiv-mutation}
Let $Q$ be a cluster quiver and let $k\in Q_0$ be a vertex. The \emph{mutation} of $Q$ at $k$ is the quiver $\mu_kQ$ obtained from $Q$ via the following procedure.
\begin{enumerate}
\item For each length $2$ path $i\longrightarrow k\longrightarrow j$, add an arrow $i\longrightarrow j$.
\item Reverse the direction of all arrows incident with $k$.
\item Choose a maximal set of $2$-cycles, and remove all arrows appearing in them.
\end{enumerate}
\end{defn}

\begin{eg}
We mutate the $4$-vertex quiver from the above list of examples at vertex $1$.
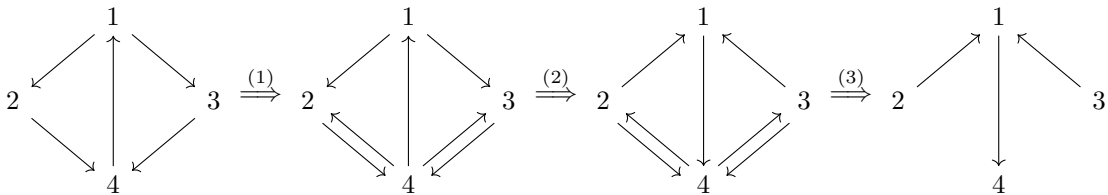
\[\begin{tikzcd}[column sep=17pt,row sep=15pt]
&1\arrow{dl}\arrow{dr}\\
2\arrow{dr}&&3\arrow{dl}\\
&4\arrow{uu}
\end{tikzcd}
\stackrel{\text{(1)}}{\Longrightarrow}
\begin{tikzcd}[column sep=17pt,row sep=15pt]
&1\arrow{dl}\arrow{dr}\\
2\arrow[shift right]{dr}&&3\arrow[shift left]{dl}\\
&4\arrow{uu}\arrow[shift right]{ul}\arrow[shift left]{ur}
\end{tikzcd}
\stackrel{\text{(2)}}{\Longrightarrow}
\begin{tikzcd}[column sep=17pt,row sep=15pt]
&1\arrow{dd}\\
2\arrow[shift right]{dr}\arrow{ur}&&3\arrow[shift left]{dl}\arrow{ul}\\
&4\arrow[shift right]{ul}\arrow[shift left]{ur}
\end{tikzcd}
\stackrel{\text{(3)}}{\Longrightarrow}
\begin{tikzcd}[column sep=17pt,row sep=15pt]
&1\arrow{dd}\\
2\arrow{ur}&&3\arrow{ul}\\
&4\end{tikzcd}\]
For readers wishing to experiment with further examples, we recommend Keller's Java applet \cite{Keller-Java}.
\end{eg}

It is straightforward to check that $\mu_k(\mu_kQ))=Q$, that is, mutating twice at the same vertex recovers the original quiver. Using this mutation operation, we can associate to each cluster quiver a cluster algebra, in the following way.

\begin{defn}
Let $\QQ(x_1,\dotsc,x_n)$ be the field of rational functions in variables $x_i$ for $i\in\{1,\dotsc,n\}$. A \emph{seed} consists of a cluster quiver $Q$ with vertex set $Q_0=\{1,\dotsc,n\}$, and a free generating set $\{f_1,\dotsc,f_n\}\subseteq\QQ(x_1,\dotsc,x_n)$ indexed by this set of vertices---being a free generating set means that the smallest subfield of $\QQ(x_1,\dotsc,x_n)$ containing $\QQ$ and the set $\{f_1,\dotsc,f_n\}$ is $\QQ(x_1,\dotsc,x_n)$ itself.

We can extend mutation to an operation on seeds. Given a seed $(Q,\{f_i\})$ and $k\in Q_0$, we define $\mu_k(Q,\{f_i\})=(\mu_kQ,\{f_i'\})$ where $\mu_kQ$ is the mutated quiver as in Definition~\ref{d:quiv-mutation}, and
\begin{equation}
\label{eq:cv-mutation}
f_i'=\begin{cases}f_i,&i\ne k,\\[2ex]\displaystyle\frac{1}{f_k}\Big(\prod_{k\to j}f_j+\prod_{\ell\to k}f_\ell\Big),&i=k.\end{cases}
\end{equation}
In the second case, the products are over arrows with tail, respectively head, $k$. These products can be over the empty set, in which case they evaluate to $1$. Note that mutating twice at the same vertex recovers the original seed.

Given a cluster quiver $Q$ with vertices $Q_0=\{1,\dotsc,n\}$, we consider the \emph{initial seed} $s_0=(Q,\{x_i\})$, whose functions are given by the distinguished generators of $\QQ(x_1,\dotsc,x_n)$. Let $\mathcal{S}_Q$ be the set of all seeds obtained from $s_0$ by a finite sequence of mutations. The \emph{cluster algebra} $\clustalg{Q}$ of $Q$ is the $\QQ$-subalgebra of $\QQ(x_1,\dotsc,x_n)$ generated by all functions appearing in all seeds in $\mathcal{S}_Q$.

Each function appearing in a seed in $\mathcal{S}_Q$ is called a \emph{cluster variable} of $\clustalg{Q}$, and the set $\{f_i\}$ of functions appearing in a single seed $(Q',\{f_i\})\in\mathcal{S}_Q$ is called a \emph{cluster} of $\clustalg{Q}$.
\end{defn}

Despite this somewhat esoteric definition, cluster algebras have made surprising appearances in a number of areas of mathematics, as we observed in the introduction. Notably, the coordinate rings of many important algebraic varieties, such as Grassmannians and other varieties of flags, are isomorphic to cluster algebras, at least after extending the definition slightly to replace $\QQ$ by a field extension (typically $\CC$) and choosing a set of frozen vertices in $Q_0$, mutations at which are not permitted when constructing the set $\mathcal{S}_Q$ of seeds---this means in particular that the variable $x_i$ appears in every cluster when $i$ is a frozen vertex.

The main conclusion of this section will be that for certain cluster algebras, the cluster variables give formulae for the entries of a frieze pattern in terms of the entries in a lightning bolt, as the following simple example demonstrates.

\begin{eg}
\label{eg:a2}
Let $Q$ be the quiver $1\longrightarrow 2$. Then we can compute that the seeds of $\clustalg{Q}$ are
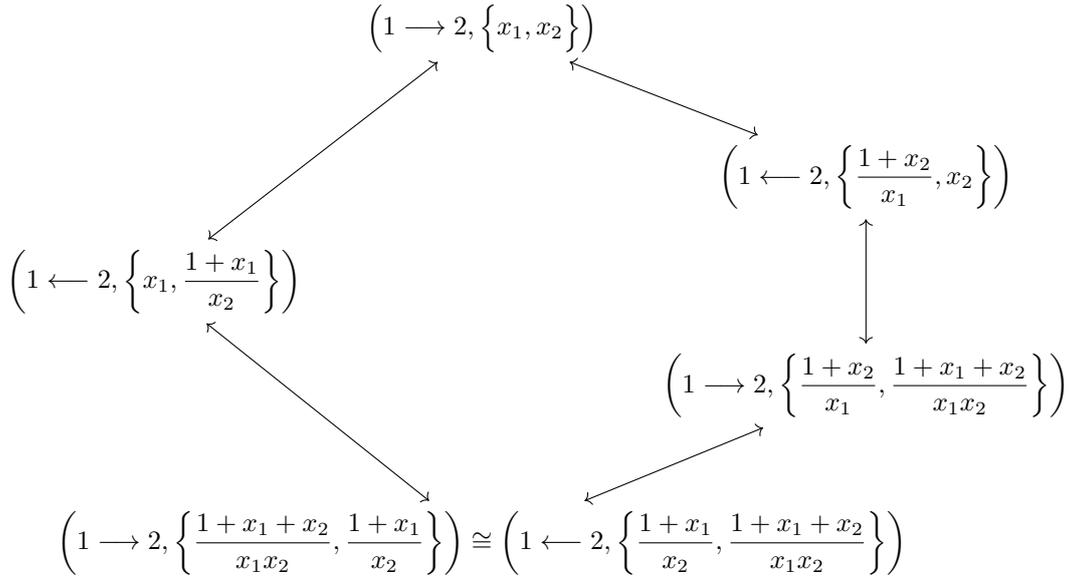
\begin{center}
\resizebox{0.95\textwidth}{!}{$\begin{tikzcd}[column sep=-120pt,ampersand replacement=\&]
\&\left(1\longrightarrow 2,\left\{x_1,x_2\vphantom{\frac{1+x_2}{x_1}}\right\}\right)\arrow[leftrightarrow]{dr}\\[10pt]
\&\&\left(1\longleftarrow 2,\left\{\dfrac{1+x_2}{x_1},x_2\right\}\right)\arrow[leftrightarrow]{dd}\\[-10pt]
\left(1\longleftarrow 2,\left\{x_1,\dfrac{1+x_1}{x_2}\right\}\right)\arrow[leftrightarrow]{uur}\\[-10pt]
\&\&\left(1\longrightarrow 2,\left\{\dfrac{1+x_2}{x_1},\dfrac{1+x_1+x_2}{x_1x_2}\right\}\right)\arrow[leftrightarrow]{dl}\\[10pt]
\&\left(1\longrightarrow 2,\left\{\dfrac{1+x_1+x_2}{x_1x_2},\dfrac{1+x_1}{x_2}\right\}\right)\cong\left(1\longleftarrow 2,\left\{\dfrac{1+x_1}{x_2},\dfrac{1+x_1+x_2}{x_1x_2}\right\}\right)\arrow[leftrightarrow]{uul}
\end{tikzcd}$}
\end{center}
In this diagram, each two-headed arrow represents a mutation. The `isomorphism' on the final line indicates that the two seeds are related by relabelling the quiver vertices, and correspondingly re-ordering the cluster variables. (If we suppressed this labelling by writing each cluster variable directly on the corresponding quiver vertex, both seeds would consist of an arrow from $\frac{1+x_1+x_2}{x_1x_2}$ to $\frac{1+x_1}{x_2}$.) We see that there are five cluster variables, which are precisely the five Laurent polynomials that appeared in our experiment at the end of Section~\ref{s:friezes}.
\end{eg}

We now give two important theorems concerning cluster algebras in general, which will shed light on the phenomena concerning frieze patterns from Section~\ref{s:friezes}, at least once we have understood how frieze pattern entries are given by specialising cluster variables. The first is the \emph{Laurent phenomenon}, which will completely explain Phenomenon~\ref{ph:lightning-bolt}.

\begin{thm}
\label{t:laurent-phenomenon}
Let $Q$ be a cluster quiver. Then every cluster variable in $\clustalg{Q}$ is a Laurent polynomial in the initial variables $\{x_1,\dotsc,x_n\}$.
\end{thm}
\begin{proof}
A combinatorial proof is given by Fomin and Zelevinsky \cite[Thm.~3.1]{FomZel-CA1} in the original paper defining cluster algebras (in more generality than here). A deeper geometric argument, which involves realising cluster variables as regular functions on a certain algebraic space defined as a union of tori indexed by the set of seeds, is given by Gross, Hacking and Keel \cite[Cor.~3.11]{grossbirational}.
\end{proof}

Since every cluster $\{f_i\}$ in $A_Q$ is a free generating set for $\QQ(x_1,\dotsc,x_n)$, any element of this field has a unique expression as a rational function in the $f_i$. By change of variables, it follows from Theorem~\ref{t:laurent-phenomenon} that the cluster variables of $A_Q$ are given by Laurent polynomials when written as rational functions in any of the clusters, not just the initial cluster $\{x_1,\dotsc,x_n\}$.

For most quivers $Q$, the set $\mathcal{S}_Q$ of seeds will be infinite. The second important theorem we mention here concerns when, as in Example~\ref{eg:a2}, the number of seeds (and therefore the number of cluster variables) is finite, in which case we say that the cluster algebra has finite type.

\begin{thm}[{\cite[Thm.~1.4]{FomZel-CA2}}]
\label{t:finite-type}
A cluster algebra $\clustalg{Q}$ has finite type if and only if $Q$ is related by a sequence of mutations to a quiver obtained by choosing an orientation of one of the following graphs.
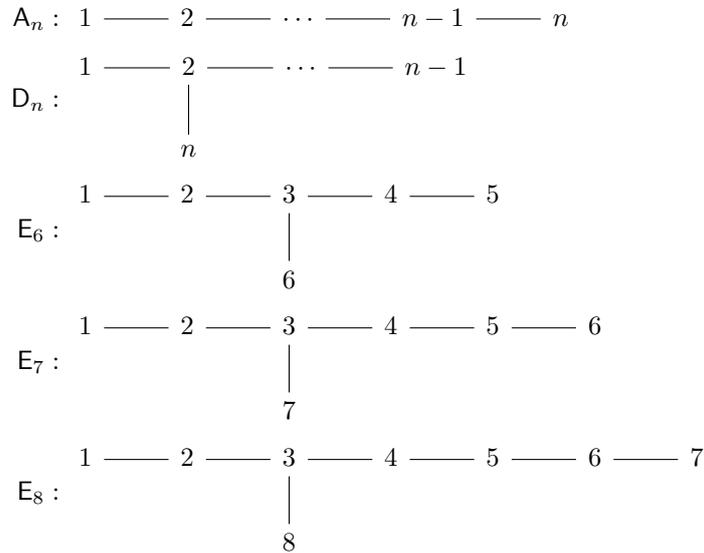
\begingroup\allowdisplaybreaks\begin{align*}
\mathsf{A}_n&:
\begin{tikzcd}[row sep=13pt,ampersand replacement=\&]
1\arrow[no head]{r}\&2\arrow[no head]{r}\&\cdots\arrow[no head]{r}\&n-1\arrow[no head]{r}\&n
\end{tikzcd}\\[\smallskipamount]
\mathsf{D}_n&:
\begin{tikzcd}[row sep=13pt,ampersand replacement=\&]
1\arrow[no head]{r}\&2\arrow[no head]{r}\&\cdots\arrow[no head]{r}\&n-1\\
\&n\arrow[no head]{u}
\end{tikzcd}\\[\smallskipamount]
\mathsf{E}_6&:
\begin{tikzcd}[row sep=13pt,ampersand replacement=\&]
1\arrow[no head]{r}\&2\arrow[no head]{r}\&3\arrow[no head]{r}\&4\arrow[no head]{r}\&5\\
\&\&6\arrow[no head]{u}
\end{tikzcd}\\[\smallskipamount]
\mathsf{E}_7&:
\begin{tikzcd}[row sep=13pt,ampersand replacement=\&]
1\arrow[no head]{r}\&2\arrow[no head]{r}\&3\arrow[no head]{r}\&4\arrow[no head]{r}\&5\arrow[no head]{r}\&6\\
\&\&7\arrow[no head]{u}
\end{tikzcd}\\[\smallskipamount]
\mathsf{E}_8&:
\begin{tikzcd}[row sep=13pt,ampersand replacement=\&]
1\arrow[no head]{r}\&2\arrow[no head]{r}\&3\arrow[no head]{r}\&4\arrow[no head]{r}\&5\arrow[no head]{r}\&6\arrow[no head]{r}\&7\\
\&\&8\arrow[no head]{u}
\end{tikzcd}
\end{align*}\endgroup
\end{thm}

Some readers may recognise the graphs appearing in Theorem~\ref{t:finite-type} as the simply-laced Dynkin diagrams, and indeed we call their orientations \emph{Dynkin quivers}. These readers may be interested to know that the number of cluster variables in the cluster algebra of a Dynkin quiver is equal to the number of positive roots of the corresponding root system, plus $n$ (the number of vertices) \cite[Thm.~1.9]{FomZel-CA2}, and also that the non-simply-laced Dynkin diagrams can be made to appear by considering more general cluster algebras from weighted quivers (or skew-symmetrisable matrices, as in \cite{FomZel-CA2}). An introduction to these Lie-theoretic concepts can be found in the book of Fulton and Harris \cite{fultonrepresentation}.

We also remark that any two Dynkin quivers with the same underlying graph are related by a sequence of mutations (indeed, by mutations only at sinks and sources). Thus the cluster algebras associated to two different orientations of the same diagram are related by a change of variables, and so up to isomorphism there is one cluster algebra per Dynkin digram.

The final thing for us to observe in this section is that the the formulae expressing general entries of a height $n$ frieze in terms of the entries in a lightning bolt are given by cluster variables in a cluster algebra of type $\mathsf{A}_n$, i.e.\ the cluster algebra associated to a quiver whose underlying graph is this Dynkin diagram. Then Phenomenon~\ref{ph:lightning-bolt} will follow from the Laurent phenomenon in Theorem~\ref{t:laurent-phenomenon}. While we won't yet have an explanation for the periodicity in Phenomenon~\ref{ph:glide}, this being the topic of the next section, Theorem~\ref{t:finite-type} tells us that type $\mathsf{A}_n$ cluster algebras have only finitely many cluster variables, and so we will at least see that our frieze pattern can have only finitely many different entries.

\begin{thm}
\label{t:cluster-frieze}
Choose a lightning bolt in a frieze of height $n$. Then every entry of the frieze is obtained by taking a cluster variable of $\clustalg{Q}$, where $Q$ is an orientation of the diagram $\mathsf{A}_n$, and specialising the indeterminates $x_1,\dotsc,x_n$ to the values of the frieze in the lightning bolt.
\end{thm}
\begin{proof}
We sketch the argument. The first step is to construct a quiver $Q$ from the lightning bolt. The vertex set is $Q_0=\{1,\dotsc,n\}$ as usual, and we associate $i$ to the $i$-th row of the frieze (using the convention that the upper trivial row is row $0$, so that our quiver vertices correspond to the non-trivial rows). There is exactly one arrow $a_i$ between each pair $i,i+1$ of consecutive integers. We take $h(a_i)=i$ and $t(a_i)=i+1$ if the lightning bolt entry in row $i+1$ is to the left of that in row $i$, and $h(a)=i+1$, $t(a_i)=i$ otherwise. We then consider the initial seed attached to this quiver. As an example, for the lightning bolt given as an example in Section~\ref{s:friezes} we obtain the seed
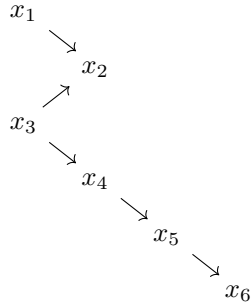
\[\begin{tikzcd}[column sep=6pt,row sep=6pt]
x_1\arrow{dr}\\&x_2\\x_3\arrow{ur}\arrow{dr}\\&x_4\arrow{dr}\\&&x_5\arrow{dr}\\&&&x_6\end{tikzcd}\]

Because $Q$ does not have any oriented cycles, it must have at least one vertex, say $k$, which is a source. In particular, there are no length $2$ paths through $k$, and so mutating at this vertex only reverses the direction of the incident arrows, making $k$ into a sink. Similarly, when computing the new cluster variable $x_k'$ in the mutated seed via formula \eqref{eq:cv-mutation}, one of the two products is empty, and we obtain
\[x_k'=\frac{1+x_{k-1}x_{k+1}}{x_k},\]
adopting the convention that $x_0=x_{n+1}=1$. In other words, the configuration
\[\begin{array}{ccc}
&x_{k-1}&\\
x_k& &x_k'\\
&x_{k+1}&
\end{array}\]
satisfies the diamond rule of a frieze pattern, and so if we specialise the $x_k$ to the entries of the frieze along our chosen lightning bolt, $x_k'$ will be specialised to the frieze entry directly to the right of that in the $k$-th row of the lightning bolt. Moreover, the quiver $\mu_kQ$ of this new seed is exactly the one coming from the lightning bolt obtained from our initial one by replacing the entry in the $k$-th row by that to its right---the fact that $k$ is a source in $Q$ corresponds to the fact that this replacement yields another valid lightning bolt.

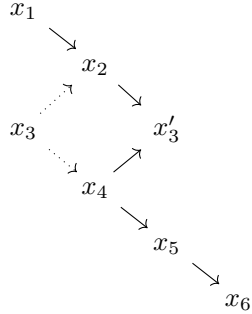
\[\begin{tikzcd}[column sep=6pt,row sep=6pt]
x_1\arrow{dr}\\&x_2\arrow{dr}\\x_3\arrow[dotted]{ur}\arrow[dotted]{dr}&&x_3'\\&x_4\arrow{ur}\arrow{dr}\\&&x_5\arrow{dr}\\&&&x_6\end{tikzcd}\]

By continuing to mutate at sources in this way, we see that all of the frieze entries to the right of the lightning bolt are specialisations of cluster variables in $\clustalg{Q}$.

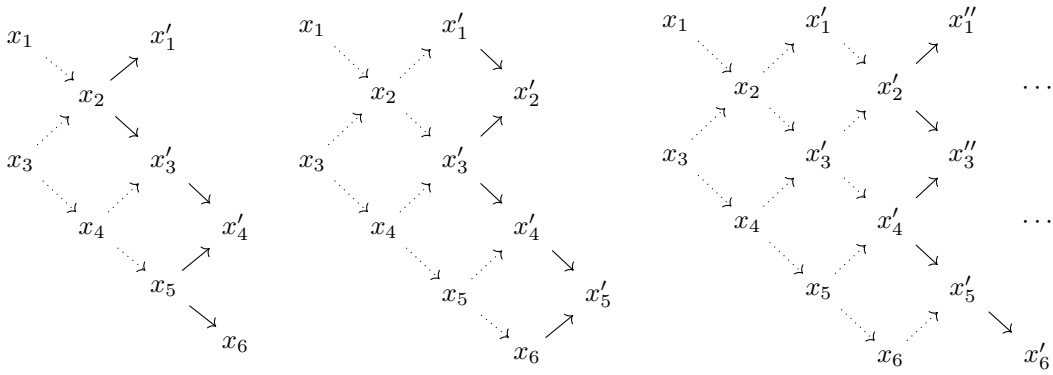
\begin{center}
\resizebox{0.95\textwidth}{!}{$\begin{tikzcd}[column sep=8pt,row sep=8pt,ampersand replacement=\&]
x_1\arrow[dotted]{dr}\&\&x_1'\\\&x_2\arrow{ur}\arrow{dr}\\x_3\arrow[dotted]{ur}\arrow[dotted]{dr}\&\&x_3'\arrow{dr}\\\&x_4\arrow[dotted]{ur}\arrow[dotted]{dr}\&\&x_4'\\\&\&x_5\arrow{ur}\arrow{dr}\\\&\&\&x_6\end{tikzcd}
\enspace
\begin{tikzcd}[column sep=8pt,row sep=8pt,ampersand replacement=\&]
x_1\arrow[dotted]{dr}\&\&x_1'\arrow{dr}\\\&x_2\arrow[dotted]{ur}\arrow[dotted]{dr}\&\&x_2'\\x_3\arrow[dotted]{ur}\arrow[dotted]{dr}\&\&x_3'\arrow{dr}\arrow{ur}\\\&x_4\arrow[dotted]{ur}\arrow[dotted]{dr}\&\&x_4'\arrow{dr}\\\&\&x_5\arrow[dotted]{ur}\arrow[dotted]{dr}\&\&x_5'\\\&\&\&x_6\arrow{ur}\end{tikzcd}
\enspace
\begin{tikzcd}[column sep=8pt,row sep=8pt,ampersand replacement=\&]
x_1\arrow[dotted]{dr}\&\&x_1'\arrow[dotted]{dr}\&\&x_1''\\\&x_2\arrow[dotted]{ur}\arrow[dotted]{dr}\&\&x_2'\arrow{ur}\arrow{dr}\&\&\cdots\\x_3\arrow[dotted]{ur}\arrow[dotted]{dr}\&\&x_3'\arrow[dotted]{ur}\arrow[dotted]{dr}\&\&x_3''\\\&x_4\arrow[dotted]{ur}\arrow[dotted]{dr}\&\&x_4'\arrow{ur}\arrow{dr}\&\&\cdots\\\&\&x_5\arrow[dotted]{ur}\arrow[dotted]{dr}\&\&x_5'\arrow{dr}\\\&\&\&x_6\arrow[dotted]{ur}\&\&x_6'\end{tikzcd}$}
\end{center}

A similar argument, using mutations at sinks, gives the same result for the entries to the left of our lightning bolt.
\end{proof}

Note that we did not consider all possible mutations of our initial seed in the proof of Theorem~\ref{t:cluster-frieze}, and indeed we do not find all seeds of $\clustalg{Q}$ via mutations only at sources and sinks. On the other hand, it turns out, as we will see in the next section, that we do find all cluster variables of $\clustalg{Q}$ in this way. Indeed, the cluster algebra of type $\mathsf{A}_n$ has $\frac{1}{2}n(n+3)$ cluster variables, which is the number of elements of the fundamental domain appearing in Phenomenon~\ref{ph:glide}.

\section{Cluster categories}
\label{s:clust-cat}

In this section, we explain how representation theory can be used to study cluster algebras and, by extension, frieze patterns. This will require introducing a number of algebraic concepts relatively quickly: readers wishing to learn more about the representation theory of quivers (and other finite-dimensional algebras) are advised to consult books by Auslander--Reiten--Smalø \cite{auslanderrepresentationtheory}, Assem--Simson--Skowroński \cite{assemelements} or Schiffler \cite{schifflerquiver}, the latter also including a section on cluster-tilted algebras, whose original definition is related to the ideas presented in these notes. We will be particularly brief about the derived category, since we will need only to understand a very special example in any detail. A slightly less brief explanation, in a bit more generality, can be found in Appendix~\ref{s:bdcat}, and a reader looking for a detailed exposition of this extremely useful and powerful construction is advised to consult Keller's survey \cite{kellerderiveduses}, or the relevant sections of books by Happel \cite{Happel-Book} or Gelfand and Manin \cite{gelfandmethods}.

A \emph{category} $\mathcal{C}$ consists of a set of objects $\Ob(\mathcal{C})$, a set $\Hom_{\mathcal{C}}(X,Y)$ of morphisms $X\to Y$ for any pair of objects $X,Y\in\Ob(\mathcal{C})$, and an associative composition law consisting of maps $\circ\colon\Hom_{\mathcal{C}}(Y,Z)\times\Hom_{\mathcal{C}}(X,Y)\to\Hom_{\mathcal{C}}(X,Z)$ for each triple $X,Y,Z\in\Ob(\mathcal{C})$; we write $g\circ f=\circ(g,f)$. Each object $X$ has an identity morphism $1_X$, such that $1_X\circ f=f$ and $g\circ 1_X=g$ whenever these compositions are defined. All of our categories will be $\KK$-linear, for $\KK$ a fixed arbitrary field, meaning that the sets $\Hom_{\mathcal{C}}(X,Y)$ are $\KK$-vector spaces, and the composition maps are $\KK$-bilinear. We will, as is common, write $X\in\mathcal{C}$ as shorthand for $X\in\Ob(\mathcal{C})$.

Given categories $\mathcal{C}$ and $\mathcal{D}$, a \emph{functor} $F\colon\mathcal{C}\to\mathcal{D}$ consists of a map $F\colon\Ob(\mathcal{C})\to\Ob(\mathcal{D})$ and, for each pair of objects $X,Y\in\mathcal{C}$, a ($\KK$-linear) map $F\colon\Hom_{\mathcal{C}}(X,Y)\to\Hom_{\mathcal{D}}(FX,FY)$. These maps have the properties that $F(1_X)=1_{FX}$ for all $X\in\mathcal{C}$ and $F(g\circ f)=F(g)\circ F(f)$ for any composable morphisms $f$ and $g$ in $\mathcal{C}$. A functor $F\colon\mathcal{C}\to\mathcal{D}$ is an \emph{equivalence} if it admits an `inverse', a functor $G\colon\mathcal{D}\to\mathcal{C}$ such that both $F\circ G$ and $G\circ F$ are naturally isomorphic \cite[\S A.2]{assemelements} to identity functors (on $\mathcal{D}$ and $\mathcal{C}$ respectively). This means in particular that $G(F(X))\iso X$ for any object $X\in\mathcal{C}$, and $F(G(Y))\iso Y$ for any object $Y\in\mathcal{D}$. An equivalence $F\colon\mathcal{C}\to\mathcal{C}$ is called an \emph{autoequivalence} of $\mathcal{C}$.

The goal of this section will be to define, for any acyclic quiver $Q$, a category $\clustcat{Q}$, called the cluster category of $Q$, whose properties reflect those of the cluster algebra $\clustalg{Q}$. These categories were originally described by Buan, Marsh, Reineke, Reiten and Todorov \cite{BMRRT} (and by Caldero, Chapoton and Schiffler in some special cases \cite{calderoquivers1,calderoquivers2}). The upshot for us will be that, when $Q$ is an orientation of the graph $\mathsf{A}_n$, we can see a height $n$ frieze pattern as a function from the set $\indec{\clustcat{Q}}$ of isomorphism classes of indecomposable objects of $\clustcat{Q}$ to the positive integers. This will explain the periodicity in Phenomenon~\ref{ph:glide}, and tell us how to construct all height $n$ friezes, not only those arising from lightning bolts.

The construction uses the representation theory of the quiver $Q$. In the case that $Q$ has underlying graph $\mathsf{A}_n$, relevant to friezes, we will also give a different and more combinatorial explanation, that may be easier to follow for readers not already familiar with quiver representations.

\begin{defn}
\label{d:quiver-rep}
A representation $(V,f)$ (often abbreviated to just $V$) of a quiver $Q$ is an assignment of a $\KK$-vector space $V_i$ to each vertex $i\in Q_0$ and a linear map $f_a\colon V_{t(a)}\to V_{h(a)}$ to each arrow $a\in Q_1$. Given representations $(V,f)$ and $(W,g)$ of $Q$, a morphism $\varphi\colon(V,f)\to(W,g)$ consists of a linear map $\varphi_i\colon V_i\to W_i$ for each $i$ in $Q_0$, such that the diagram
\[\begin{tikzcd}
V_{t(a)}\arrow{r}{f_a}\arrow{d}[swap]{\varphi_{t(a)}}&V_{h(a)}\arrow{d}{\varphi_{h(a)}}\\
W_{t(a)}\arrow{r}[swap]{g_a}&W_{h(a)}
\end{tikzcd}\]
commutes for any $a\in Q_1$. The morphism $\varphi$ is an isomorphism if every $\varphi_i$ is an isomorphism of vector spaces.

We may define the direct sum of two representations pointwise, i.e.\ by $(V\oplus W)_i=V_i\oplus W_i$, with morphisms
\[\begin{pmatrix}f_a&0\\0&g_a\end{pmatrix}\colon V_{t(a)}\oplus W_{t(a)}\to V_{h(a)}\oplus W_{h(a)}.\]
A representation is \emph{indecomposable} if is non-zero (i.e.\ some $V_i$ is not the zero vector space) and not isomorphic to the direct sum of two non-zero representations.

In these notes, we restrict to finite-dimensional representations, meaning that the vector spaces attached to quiver vertices are all finite-dimensional. We denote by $\rep{Q}$ the category whose objects are such representations of $Q$, and whose morphisms are those described above. This category is \emph{abelian}: among other things, this means that it has direct sums as defined above, its morphism spaces are abelian groups (even $\KK$-vector spaces), its morphisms have well-defined kernels and cokernels, every injective morphism is a kernel, and every surjective morphism is a cokernel.
\end{defn}

\begin{rem}
\label{r:pathcat}
Readers familiar with category theory may recognise the commutative diagram in Definition~\ref{d:quiver-rep} as similar to that involved in the definition of a natural transformation of functors \cite[\S A.2]{assemelements}. Indeed, $Q$ determines a path category, with objects the quiver vertices, morphisms $i$ to $j$ given by the directed paths from $i$ to $j$, and composition given by concatenating paths. A representation of $Q$ is then nothing but a covariant functor from this path category to the category of $\KK$-vector spaces, and a morphism of representations is a natural transformation between two such functors.
\end{rem}

The next step towards our desired category $\clustcat{Q}$ is to construct the bounded derived category $\bdcat{Q}$ of $\rep{Q}$. This construction, first introduced by Verdier \cite{verdiercategories1}, can be made for any abelian category \cite{Verdier-These} (and even more general categories) but is somewhat complicated. As a result, we give an ad hoc definition in the case of quiver representations, exploiting special properties of the abelian category $\rep{Q}$ (most importantly that it is hereditary, meaning that every subobject of a projective object \cite[Def.~I.5.2]{assemelements} is again projective). The more general construction, and its relationship to the definition given below, is briefly described in Appendix~\ref{s:bdcat}.

\begin{defn}
\label{d:bdcat}
Given a quiver $Q$, the \emph{bounded derived category} $\bdcat{Q}$ of $Q$ is defined as follows. For each $i\in\ZZ$ and $V\in\rep{Q}$, we introduce the formal symbol $\Sigma^iV$, and take the objects of $\bdcat{Q}$ to be formal direct sums of these symbols. The morphism spaces
\[\Hom_{\bdcat{Q}}(\Sigma^iV,\Sigma^jW)\defeq\Ext^{j-i}_Q(V,W)\]
are given by extension groups \cite[\S A.4]{assemelements} in $\rep{Q}$; since $\rep{Q}$ is a hereditary abelian category, this vector space may only be non-zero when $j-i=0$ or $j-i=1$. This definition of morphism spaces can be extended to the formal direct sums via the formulae
\begin{align*}
\Hom_{\bdcat{Q}}(X_1\oplus X_2,Y)&=\Hom_{\bdcat{Q}}(X_1,Y)\oplus\Hom_{\bdcat{Q}}(X_2,Y),\\
\Hom_{\bdcat{Q}}(X,Y_1\oplus Y_2)&=\Hom_{\bdcat{Q}}(X,Y_1)\oplus\Hom_{\bdcat{Q}}(X,Y_2),
\end{align*}
and the composition law is defined using cup product of extensions.

The reason for writing $\Sigma^iV$ to encode the pair $(i,V)$ is that the category $\bdcat{Q}$ carries an autoequivalence $\Sigma\colon\bdcat{Q}\to\bdcat{Q}$, whose action on objects is exactly as suggested by the notation. On morphisms, $\Sigma$ acts as the identity, noting that
\[\Hom_{\bdcat{Q}}(\Sigma^{i+1}V,\Sigma^{j+1}W)=\Ext^{j-i}_Q(V,W)=\Hom_{\bdcat{Q}}(\Sigma^iV,\Sigma^jW).\]
\end{defn}

One can give $\bdcat{\mathcal{A}}$ the structure of a triangulated category, in which $\Sigma$ is the suspension functor, part of the defining data of this structure. We will not discuss the general definition or theory of triangulated categories, but refer the interested reader to Happel's book \cite{Happel-Book}.

When $Q$ is a Dynkin quiver, we can describe $\bdcat{Q}$ in very combinatorial terms. First we associate to $Q$ an infinite quiver $\ZZ Q$, as follows. The vertices of $\ZZ Q$ are pairs $(i,n)$ with $i\in Q_0$ and $n\in\ZZ$. The arrows of $\ZZ Q$ are $a_n\colon(t(a),n)\to(h(a),n)$ and $a_n^*\colon(h(a),n)\to(t(a),n+1)$ for each $a\in Q_1$ and $n\in\ZZ$. This quiver is highly symmetric---most useful to us is the translation symmetry $\tau$, defined by $\tau(i,n)=(i,n-1)$, $\tau(a_n)=a_{n-1}$ and $\tau(a_n^*)=a_{n-1}^*$.

Each vertex $(i,n)$ of $\ZZ Q$ gives rise to the \emph{mesh relation}
\[\sum_{a:h(a)=i}a_{n+1}a_n^* - \sum_{b:t(a)=i}b_n^*b_n,\]
a formal linear combination of paths in $\ZZ Q$ (which we read from right-to-left, like composition of functions). For example, if $Q$ is the quiver
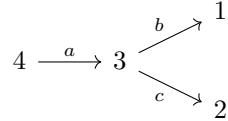
\[\begin{tikzcd}[row sep=5pt]
&&1\\
4\arrow{r}{a}&3\arrow{dr}[swap]{c}\arrow{ur}{b}&\\
&&2\end{tikzcd}\]
then $\ZZ Q$ has the local configuration
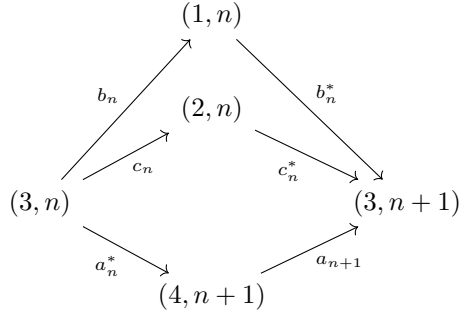
\[\begin{tikzcd}&(1,n)\arrow{ddr}{b_n^*}\\
&(2,n)\arrow{dr}[swap]{c_n^*}\\
(3,n)\arrow{uur}{b_n}\arrow{ur}[swap]{c_n}\arrow{dr}[swap]{a_n^*}&&(3,n+1)\\
&(4,n+1)\arrow{ur}[swap]{a_{n+1}}
\end{tikzcd}\]
for each $n\in\ZZ$, and the mesh relation is $a_{n+1}a_n^*-b_n^*b_n-c_n^*c_n$.

Now we define a category $\mathcal{D}_Q$ whose objects are formal direct sums of vertices of $\ZZ Q$. The set $\Hom_{\mathcal{D}_Q}((i,n),(j,m))$ of morphisms from a vertex $(i,n)$ to a vertex $(j,m)$ is the vector space spanned by paths from $(i,n)$ to $(j,m)$, subject to the mesh relations: this means that a linear combination of paths obtained from a mesh relation by postcomposing all terms with a fixed path $p$ and precomposing all terms with a fixed path $q$ is $0$. For example, if we see the configuration
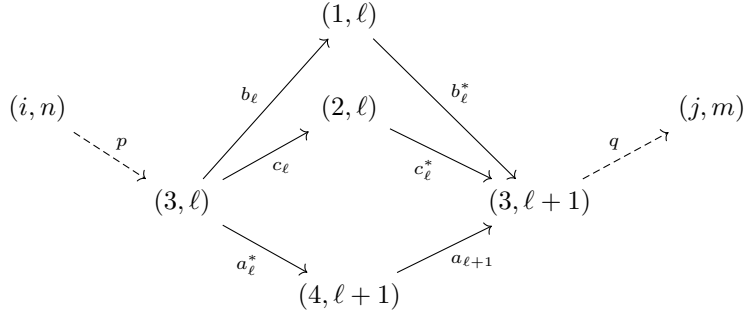
\[\begin{tikzcd}&&(1,\ell)\arrow{ddr}{b_\ell^*}\\
(i,n)\arrow[dashed]{dr}{p}&&(2,\ell)\arrow{dr}[swap]{c_\ell^*}&&(j,m)\\
&(3,\ell)\arrow{uur}{b_\ell}\arrow{ur}[swap]{c_\ell}\arrow{dr}[swap]{a_\ell^*}&&(3,\ell+1)\arrow[dashed]{ur}{q}\\
&&(4,\ell+1)\arrow{ur}[swap]{a_{\ell+1}}
\end{tikzcd}\]
for some paths $p$ and $q$, then the equation
\[qa_{\ell+1}a_\ell^*p-qb_\ell^*b_\ell p-qc_\ell^*c_\ell p=q(a_{\ell+1}a_\ell-b_\ell^*b_\ell-c_\ell^*c_\ell)p=0\]
holds in $\Hom_{\mathcal{D}_Q}((i,n),(j,m))$. We extend this definition of morphisms to all objects via the rules
\begin{align*}
\Hom_{\mathcal{D}_Q}(v_1\oplus v_2,w)&=\Hom_{\mathcal{D}_Q}(v_1,w)\oplus\Hom_{\mathcal{D}_Q}(v_2,w),\\
\Hom_{\mathcal{D}_Q}(v,w_1\oplus w_2)&=\Hom_{\mathcal{D}_Q}(v,w_1)\oplus\Hom_{\mathcal{D}_Q}(v,w_2).
\end{align*}
The symmetry $\tau$ of $\ZZ Q$ takes each mesh relation to another mesh relation, and so induces an autoequivalence of $\mathcal{D}_Q$.

\begin{thm}[{\cite[\S I.5.6]{Happel-Book}}]
\label{t:comb-dcat}
When $Q$ is a Dynkin quiver, there is an equivalence of categories $\mathcal{D}_Q\isoto\bdcat{Q}$.
\end{thm}

This statement is perhaps a bit misleading for readers already familiar with representations of quivers. Indeed, the most natural equivalence of categories is $\mathcal{D}_{Q^{\op}}\isoto\bdcat{Q}$, where $Q^{\op}$ is the opposite quiver of $Q$, obtained by reversing the directions of all the arrows, and takes the object $(i,0)\in\mathcal{D}_Q$ to $\Sigma^0P_i\in\bdcat{Q}$, where $P_i$ is the indecomposable projective representation at vertex $i$. However, by drawing some examples (or looking at the $\mathsf{A}_6$ case shown below) the reader will quickly convince themselves that $\mathcal{D}_Q$ is independent of the orientation of $Q$, up to equivalence of categories, and so the theorem as stated is also true. In our application, it will be enough to know that we can substitute $\bdcat{Q}$ for $\mathcal{D}_Q$ when $Q$ is a quiver of type $\mathsf{A}_n$, and the details of the functor providing the equivalence will not be relevant.

The following picture shows $\mathcal{D}_Q$ in the case that $Q$ is the orientation of $\mathsf{A}_6$ used to illustrate the proof of Theorem~\ref{t:cluster-frieze}.
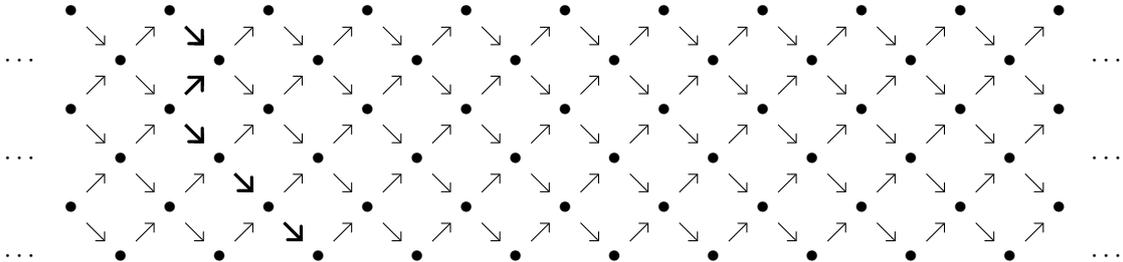
\begin{center}
\resizebox{0.95\textwidth}{!}{\begin{tikzpicture}[scale=0.75]
\foreach \x in {1,2,3,4,5,6,7,8,9,10,11}
   \foreach \y in {2,4,6}
      \draw (2*\x,\y) node ({\x}c{\y}) {$\bullet$};
\foreach \x in {1,2,3,4,5,6,7,8,9,10}
   \foreach \y in {1,3,5}
      \draw (2*\x+1,\y) node ({\x}c{\y}) {$\bullet$};
\foreach \x in {0,11}
   \foreach \y in {1,3,5}
      \draw (2*\x+1,\y) node {$\cdots$};
\foreach \x in {1,2,3,4,5,6,7,8,9,10}
   \foreach \y/\z in {2/1,4/3,6/5}
      \draw[->] ({\x}c{\y})--({\x}c{\z});
\foreach \x in {1,2,3,4,5,6,7,8,9,10}
   \foreach \y/\z in {2/3,4/5}
      \draw[->] ({\x}c{\y})--({\x}c{\z});
\foreach \x/\w in {1/2,2/3,3/4,4/5,5/6,6/7,7/8,8/9,9/10,10/11}
   \foreach \y/\z in {1/2,3/4,5/6}
      \draw[->] ({\x}c{\y})--({\w}c{\z});
\foreach \x/\w in {1/2,2/3,3/4,4/5,5/6,6/7,7/8,8/9,9/10,10/11}
   \foreach \y/\z in {3/2,5/4}
      \draw[->] ({\x}c{\y})--({\w}c{\z});
\foreach \x/\y/\z/\w in {2/6/2/5,2/4/2/5,2/4/2/3,2/3/3/2,3/2/3/1}
\draw[->, very thick] ({\x}c{\y}) -- ({\z}c{\w});
\end{tikzpicture}}
\end{center}
In this picture, each point $\bullet$ represents an indecomposable object in $\mathcal{D}_Q$, i.e.\ one of the points $(i,n)$, and each arrow represents an irreducible morphism, a non-isomorphism between indecomposable objects which is not expressible as a product of such morphisms. The arrows between the vertices $(i,0)$ are shown in bold, making the initial quiver $Q$ visible. Comparing to Section~\ref{s:friezes}, we see that it is natural to think of a height $n$ frieze pattern (excluding the trivial rows) as a function on the indecomposable objects of $\mathcal{D}_Q$ (or equivalently those of $\bdcat{Q}$, by Theorem~\ref{t:comb-dcat}) for any orientation $Q$ of the Dynkin diagram $\mathsf{A}_n$.

The mesh relations in $\mathcal{D}_Q$ say that each square
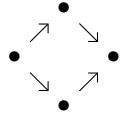
\[\begin{tikzpicture}[scale=0.8]
\foreach \x/\y in {0/1,1/2,1/0,2/1}
   \draw (\x,\y) node ({\x}c{\y}) {$\bullet$};
\foreach \x/\y/\z/\w in {0/1/1/2,0/1/1/0,1/2/2/1,1/0/2/1}
   \draw[->] ({\x}c{\y}) -- ({\z}c{\w});
\end{tikzpicture}\]
either commutes or anti-commutes, whereas at the bottom of the diagram, a pair of morphisms
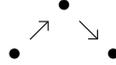
\[\begin{tikzpicture}[scale=0.8]
\foreach \x/\y in {0/1,1/2,2/1}
   \draw (\x,\y) node ({\x}c{\y}) {$\bullet$};
\foreach \x/\y/\z/\w in {0/1/1/2,1/2/2/1}
   \draw[->] ({\x}c{\y}) -- ({\z}c{\w});
\end{tikzpicture}\]
composes to zero (and similarly at the top). In type $\mathsf{A}_n$, it is not too difficult to compute the space of morphisms between any pair of objects using these rules, as in the following example.
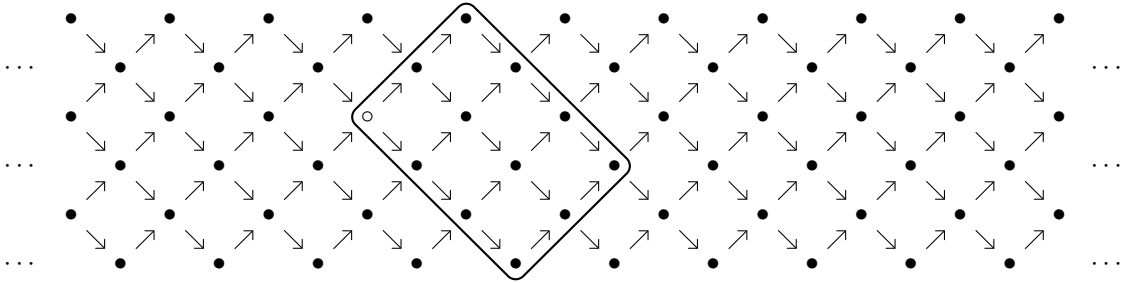
\begin{center}
\resizebox{0.95\textwidth}{!}{\begin{tikzpicture}[scale=0.75]
\foreach \x in {1,2,3,5,6,7,8,9,10,11}
   \foreach \y in {2,4,6}
      \draw (2*\x,\y) node ({\x}c{\y}) {$\bullet$};
\draw (8,2) node ({4}c{2}) {$\bullet$};
\draw (8,4) node ({4}c{4}) {$\circ$};
\draw (8,6) node ({4}c{6}) {$\bullet$};
\foreach \x in {1,2,3,4,5,6,7,8,9,10}
   \foreach \y in {1,3,5}
      \draw (2*\x+1,\y) node ({\x}c{\y}) {$\bullet$};
\foreach \x in {0,11}
   \foreach \y in {1,3,5}
      \draw (2*\x+1,\y) node {$\cdots$};
\foreach \x in {1,2,3,4,5,6,7,8,9,10}
   \foreach \y/\z in {2/1,4/3,6/5}
      \draw[->] ({\x}c{\y})--({\x}c{\z});
\foreach \x in {1,2,3,4,5,6,7,8,9,10}
   \foreach \y/\z in {2/3,4/5}
      \draw[->] ({\x}c{\y})--({\x}c{\z});
\foreach \x/\w in {1/2,2/3,3/4,4/5,5/6,6/7,7/8,8/9,9/10,10/11}
   \foreach \y/\z in {1/2,3/4,5/6}
      \draw[->] ({\x}c{\y})--({\w}c{\z});
\foreach \x/\w in {1/2,2/3,3/4,4/5,5/6,6/7,7/8,8/9,9/10,10/11}
   \foreach \y/\z in {3/2,5/4}
      \draw[->] ({\x}c{\y})--({\w}c{\z});
\draw[rounded corners,thick]
(7.6,4) -- (10,6.4) -- (13.4,3) -- (11,0.6) -- cycle;
\end{tikzpicture}}
\end{center}
The rectangle shows those indecomposable objects having a non-zero morphism from the fixed object indicated by $\circ$, at the left-hand corner of the rectangle. Moreover, the space of morphisms from $\circ$ to each of these objects is $1$-dimensional, spanned by any path from $\circ$ to the given object---any two such paths determine the same morphism because of the mesh relations. In general the morphisms starting at a given indecomposable object $X\in\mathcal{D}_Q$, for $Q$ of type $\mathsf{A}_n$, can be computed by drawing a maximal rectangle with $X$ at its left-hand corner: note that when $X$ is on the upper or lower edge of the figure, this rectangle will degenerate to a line.

The autoequivalence $\tau$ acts by translating the picture one step to the left, whereas $\Sigma$ acts by a glide reflection to the right, with fundamental domain as shown.
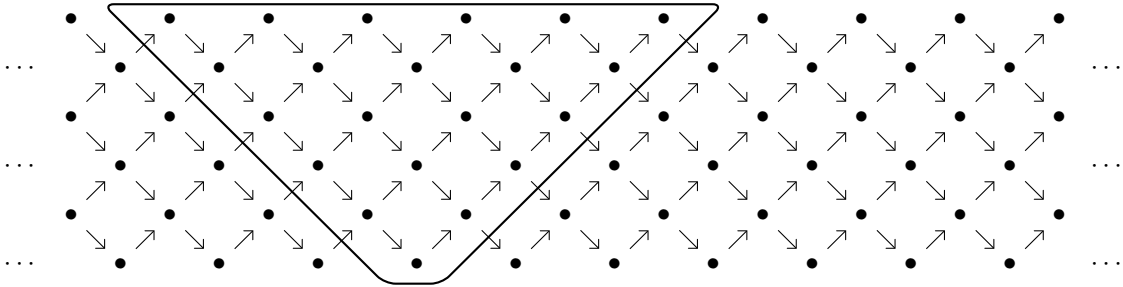
\begin{center}
\resizebox{0.95\textwidth}{!}{\begin{tikzpicture}[scale=0.75]
\foreach \x in {1,2,3,4,5,6,7,8,9,10,11}
   \foreach \y in {2,4,6}
      \draw (2*\x,\y) node ({\x}c{\y}) {$\bullet$};
\foreach \x in {1,2,3,4,5,6,7,8,9,10}
   \foreach \y in {1,3,5}
      \draw (2*\x+1,\y) node ({\x}c{\y}) {$\bullet$};
\foreach \x in {0,11}
   \foreach \y in {1,3,5}
      \draw (2*\x+1,\y) node {$\cdots$};
\foreach \x in {1,2,3,4,5,6,7,8,9,10}
   \foreach \y/\z in {2/1,4/3,6/5}
      \draw[->] ({\x}c{\y})--({\x}c{\z});
\foreach \x in {1,2,3,4,5,6,7,8,9,10}
   \foreach \y/\z in {2/3,4/5}
      \draw[->] ({\x}c{\y})--({\x}c{\z});
\foreach \x/\w in {1/2,2/3,3/4,4/5,5/6,6/7,7/8,8/9,9/10,10/11}
   \foreach \y/\z in {1/2,3/4,5/6}
      \draw[->] ({\x}c{\y})--({\w}c{\z});
\foreach \x/\w in {1/2,2/3,3/4,4/5,5/6,6/7,7/8,8/9,9/10,10/11}
   \foreach \y/\z in {3/2,5/4}
      \draw[->] ({\x}c{\y})--({\w}c{\z});
\draw[rounded corners,thick]
(2.65,6.3) -- (15.2,6.3) -- (9.5,0.6) -- (8.35,0.6) -- cycle;
\end{tikzpicture}}
\end{center}

A consequence of Theorem~\ref{t:comb-dcat} is that the autoequivalence $\tau$ of $\mathcal{D}_Q$ can be seen as an autoequivalence of $\bdcat{Q}$. In fact, this coincides with an autoequivalence of $\bdcat{Q}$ defined intrinsically for any acyclic quiver $Q$, not just the Dynkin quivers---it is closely related to a functor on $\rep{Q}$ called the Auslander--Reiten translation, also typically denoted by $\tau$. (For readers already familiar with this functor on $\rep{Q}$; in $\bdcat{Q}$, the equivalence $\tau$ takes $\Sigma^nV$ to $\Sigma^n(\tau V)$ when $V$ is indecomposable and not projective, and takes $\Sigma^nP_i$ to $\Sigma^{n-1}I_i$, where $P_i$, respectively $I_i$, denotes the indecomposable projective, respectively injective, representation at vertex $i$.)

The two autoequivalences $\Sigma$ and $\tau$ of $\bdcat{Q}$ even commute with each other: $\Sigma\circ\tau=\tau\circ\Sigma$. We are finally ready to define the category $\clustcat{Q}$.

\begin{defn}[{\cite[\S1]{BMRRT}}]
Given an acyclic quiver $Q$, its \emph{cluster category} $\clustcat{Q}$ is the orbit category
\[\clustcat{Q}=\bdcat{Q}/(\Sigma^{-1}\circ\tau)\]
for the action of $\Sigma^{-1}\circ\tau$ on $\bdcat{Q}$. By definition, this has the same objects as $\bdcat{Q}$, but the morphism spaces are
\begin{equation}
\label{eq:orb-cat-morphisms}
\Hom_{\clustcat{Q}}(X,Y)=\bigoplus_{n\in\ZZ}\Hom_{\bdcat{Q}}(X,(\Sigma^{-1}\circ\tau)^nY).
\end{equation}
For each $X,Y\in\bdcat{Q}$, only finitely many of the vector spaces appearing on the right-hand side of \eqref{eq:orb-cat-morphisms} are non-zero, and so the morphism spaces in $\clustcat{Q}$ are still finite-dimensional vector spaces. The composition
\[\circ\colon\Hom_{\clustcat{Q}}(Y,Z)\times\Hom_{\clustcat{Q}}(X,Y)\to\Hom_{\clustcat{Q}}(X,Z)\]
has components given by the maps
\begin{align*}
\Hom_{\bdcat{Q}}(Y,(\Sigma^{-1}\circ\tau)^nZ){\times}\Hom_{\bdcat{Q}}(X,(\Sigma^{-1}\circ\tau)^m&Y)\\
&\to\Hom_{\bdcat{Q}}(X,(\Sigma^{-1}\circ\tau)^{n+m}Z),\\
(g,f)&\mapsto(\Sigma^{-1}\circ\tau)^m(g)\circ f,
\end{align*}
noting that
\[(\Sigma^{-1}\circ\tau)^m\colon\Hom_{\bdcat{Q}}(Y,(\Sigma^{-1}\circ\tau)^nZ)\to\Hom_{\bdcat{Q}}((\Sigma^{-1}\circ\tau)^mY,(\Sigma^{-1}\circ\tau)^{n+m}Z).\]
Like $\bdcat{Q}$, the cluster category $\clustcat{Q}$ is a triangulated category, by a result of Keller \cite{Keller-Orbit}.
\end{defn}
While $\clustcat{Q}$ may in some sense appear much `bigger' than $\bdcat{Q}$---it has the same set of objects as $\bdcat{Q}$, and more morphisms between any two---in practical terms it is actually `smaller', as indecomposable objects which are non-isomorphic in $\bdcat{Q}$ can become isomorphic in $\clustcat{Q}$. Indeed, any two objects of $\bdcat{Q}$ in the same $(\Sigma^{-1}\circ\tau)$-orbit are isomorphic in $\clustcat{Q}$.  Both of the autoequivalences $\tau$ and $\Sigma$ of $\bdcat{Q}$ descend to autoequivalences of $\clustcat{Q}$, and indeed on $\clustcat{Q}$ they coincide (because $\Sigma^{-1}\circ\tau=\id{\clustcat{Q}}$ by the orbit category construction).

Returning to our $\mathsf{A}_6$ example, we see that a fundamental domain for $\Sigma^{-1}\circ\tau$ is as shown (cf.\ Phenomenon~\ref{ph:glide}).
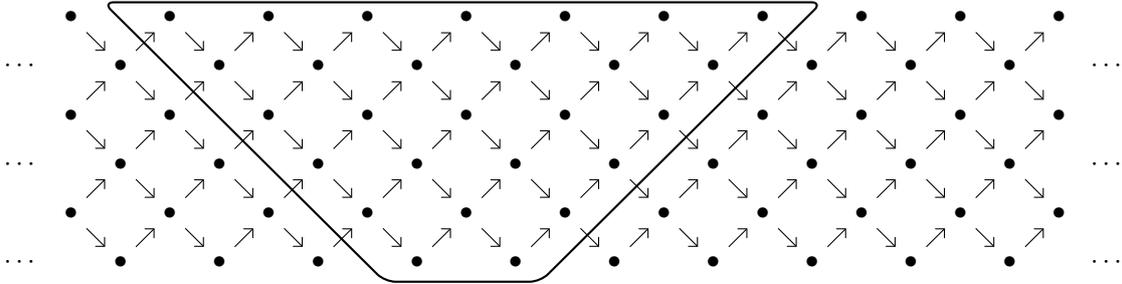
\begin{center}
\resizebox{0.95\textwidth}{!}{\begin{tikzpicture}[scale=0.75]
\foreach \x in {1,2,3,4,5,6,7,8,9,10,11}
   \foreach \y in {2,4,6}
      \draw (2*\x,\y) node ({\x}c{\y}) {$\bullet$};
\foreach \x in {1,2,3,4,5,6,7,8,9,10}
   \foreach \y in {1,3,5}
      \draw (2*\x+1,\y) node ({\x}c{\y}) {$\bullet$};
\foreach \x in {0,11}
   \foreach \y in {1,3,5}
      \draw (2*\x+1,\y) node {$\cdots$};
\foreach \x in {1,2,3,4,5,6,7,8,9,10}
   \foreach \y/\z in {2/1,4/3,6/5}
      \draw[->] ({\x}c{\y})--({\x}c{\z});
\foreach \x in {1,2,3,4,5,6,7,8,9,10}
   \foreach \y/\z in {2/3,4/5}
      \draw[->] ({\x}c{\y})--({\x}c{\z});
\foreach \x/\w in {1/2,2/3,3/4,4/5,5/6,6/7,7/8,8/9,9/10,10/11}
   \foreach \y/\z in {1/2,3/4,5/6}
      \draw[->] ({\x}c{\y})--({\w}c{\z});
\foreach \x/\w in {1/2,2/3,3/4,4/5,5/6,6/7,7/8,8/9,9/10,10/11}
   \foreach \y/\z in {3/2,5/4}
      \draw[->] ({\x}c{\y})--({\w}c{\z});
\draw[rounded corners,thick]
(2.65,6.3) -- (17.2,6.3) -- (11.5,0.6) -- (8.35,0.6) -- cycle;
\end{tikzpicture}}
\end{center}
Thus the isomorphism classes of indecomposable objects in $\clustcat{Q}$ are in bijection with the points inside this fundamental domain. Note that there are morphisms from objects at the right-hand end of this domain to those at the left-hand end, as the arrows crossing these two ends of the domain are identified---indeed one can think of $\clustcat{Q}$ as being drawn on a Möbius band, obtained as a quotient of the strip on which we draw $\mathcal{D}_Q$. If we think of frieze patterns as functions on the indecomposable objects of $\mathcal{D}_Q$, Phenomenon~\ref{ph:glide} is the observation that these functions descend to $\clustcat{Q}$, or in other words that they are constant on $(\Sigma^{-1}\circ\tau)$-orbits.

As the name suggests, the cluster category $\clustcat{Q}$ can be used to study the cluster algebra $\clustalg{Q}$: it is a categorification of this cluster algebra. More precisely, it is an additive categorification---cluster algebras may also have monoidal categorifications \cite{hernandezclusteralgebras,KKKO}, which have a rather different flavour more in common with categorifications in Lie theory, whereby the cluster algebra is realised as the Grothendieck ring of a monoidal category. The most important feature of $\clustcat{Q}$ for us is that certain isomorphism classes of objects in $\clustcat{Q}$ are in bijection with the cluster variables of $\clustalg{Q}$, whereas others are in bijection with the clusters.

\begin{defn}
We say objects $X,Y\in\clustcat{Q}$ are \emph{compatible} if $\Hom_{\clustcat{Q}}(X,\Sigma Y)=0$. An object $X\in\clustcat{Q}$ is \emph{rigid} if it is compatible with itself.

We write $\add{X}$ for the set of objects of $\clustcat{Q}$ isomorphic to direct sums of direct summands of $X$, and say that $X$ is \emph{cluster-tilting} if $\add{X}$ is precisely the set of objects compatible with $X$. Cluster-tilting objects are sometimes referred to by the longer but more descriptive name \emph{maximal $1$-orthogonal} \cite{Iyama-HART}.
\end{defn}

\begin{rem}
A consequence of the Auslander--Reiten formula for finite-dimensional algebras is that
\[\Hom_{\clustcat{Q}}(X,\Sigma Y)=\Hom_{\clustcat{Q}}(Y,\tau X)^*=\Hom_{\clustcat{Q}}(Y,\Sigma X)^*,\]
where $(-)^*$ denotes duality for $\KK$-vector spaces, so the relation of being compatible is symmetric. However, it is not reflexive, since not every object is rigid (although at least every indecomposable object is rigid when $Q$ is a Dynkin quiver), and nor is it transitive. The above formula concerning $\Hom$-spaces in $\clustcat{Q}$ expresses that $\clustcat{Q}$ is \emph{$2$-Calabi--Yau} as a triangulated category. For more appearances and uses of Calabi--Yau triangulated categories, see Keller's survey \cite{Keller-CY}.

In $\clustcat{Q}$, but not in all $2$-Calabi--Yau triangulated categories, cluster-tilting objects are precisely the maximal rigid objects, i.e.\ those rigid objects $X$ for which $X\oplus Y$ is rigid only if $Y\in\add{X}$.
\end{rem}

\begin{thm}[{\cite{calderotriangulated2}, \cite[Thm.~A.1]{BIRSc}}]
\label{t:clust-bij}
A choice of cluster-tilting object $T=\bigoplus_{i=1}^nT_i\in\clustcat{Q}$, with a decomposition into indecomposable summands $T_i$, induces a bijection $X\mapsto\varphi_X$ between the indecomposable rigid objects of $\clustcat{Q}$ and the cluster variables of $\clustalg{Q}$. Under this bijection, compatible pairs of indecomposable rigid objects are sent to cluster variables that appear together in the same cluster. In particular, there is an induced bijection between the cluster-tilting objects of $\clustcat{Q}$ and the clusters of $\clustalg{Q}$. We have $\varphi_{T_i}=x_i$, and hence the bijection sends the chosen cluster-tilting object $T$ to the initial cluster $\{x_1,\dotsc,x_n\}$.

For $Q$ of type $\mathsf{A}_n$, the cluster variables corresponding to indecomposables in a mesh satisfy the $\SL_2$ diamond rule. That is, for each configuration
\[\begin{tikzpicture}
\draw (0,1) node ({0}c{1}) {$A$};
\draw (1,2) node ({1}c{2}) {$B$};
\draw (1,0) node ({1}c{0}) {$C$};
\draw (2,1) node ({2}c{1}) {$D$};
\foreach \x/\y/\z/\w in {0/1/1/2,0/1/1/0,1/2/2/1,1/0/2/1}
   \draw[->] ({\x}c{\y}) -- ({\z}c{\w});
\end{tikzpicture}\]
the corresponding cluster variables satisfy $\varphi_A\varphi_D-\varphi_B\varphi_C=1$. At the boundary meshes, either $B$ or $C$ will be missing, and we take $\varphi_B=1$ or $\varphi_C=1$ as appropriate in the preceding formula.
\end{thm}

The map $X\mapsto\varphi_X$ appearing in Theorem~\ref{t:clust-bij} is obtained by restricting to rigid objects the assignment of a Laurent polynomial $\varphi_X$ to any object $X\in\clustcat{Q}$ via an explicit formula, known as the Caldero--Chapoton formula \cite{CalCha}; unfortunately, the ingredients in this formula require more representation theory than there is space to discuss here. The interested reader may consult the original paper \cite{CalCha}, and a survey by Plamondon \cite{plamondonclustercharacters}.

Combining Theorem~\ref{t:clust-bij} with a classical result of representation theory---Gabriel's theorem \cite{gabrielunzerlegbare1}---gives an alternative proof of Theorem~\ref{t:finite-type}. Precisely, Theorem~\ref{t:clust-bij} shows that $\clustalg{Q}$ has only finitely many cluster variables if and only if $\clustcat{Q}$ has only finitely many indecomposable objects. By construction, it is equivalent to ask that $\rep{Q}$ has finitely many indecomposable objects (since $\clustcat{Q}$ has $n=|Q_0|$ more such objects than $\rep{Q}$), and Gabriel's theorem states that this happens if and only if $Q$ is an orientation of a simply-laced Dynkin diagram.

Theorems~\ref{t:cluster-frieze} and \ref{t:clust-bij} show that we can consider a frieze pattern of height $n$ (or at least its non-trivial rows) to be a function on the set $\indec{\clustcat{Q}}$ of indecomposable objects of the cluster category $\clustcat{Q}$ for $Q$ an $\mathsf{A}_n$-quiver (all of which are rigid). Indeed, each frieze entry is the specialisation of a cluster variable, related by Theorem~\ref{t:clust-bij} to an element of $\indec{\clustcat{Q}}$. This connection between cluster categories of type $\mathsf{A}_n$ and frieze patterns was first observed by Caldero and Chapoton \cite[\S5]{CalCha}.

Our diagrams in Section~\ref{s:friezes} draw frieze patterns as functions on the set of isoclasses of indecomposable objects of the derived category $\bdcat{Q}$, but since indecomposable objects of $\bdcat{Q}$ in the same $(\Sigma^{-1}\circ\tau)$-orbit are isomorphic indecomposable objects in $\clustcat{Q}$, and hence correspond to the same cluster variable, a frieze must be constant on each of these orbits. Since $\Sigma^{-1}\circ\tau$ acts by a glide reflection, we have an explanation for Phenomenon~\ref{ph:glide}.

Our final result is a second classification of height $n$ friezes, in terms of cluster-tilting objects.

\begin{thm}
\label{t:class2}
Let $Q$ be any $\mathsf{A}_n$ quiver. Then the friezes of height $n$ are in bijection with the cluster-tilting objects of the cluster category $\clustcat{Q}$. Indeed, given a cluster tilting object $T$, there is a unique frieze pattern taking the value $1$ on each indecomposable summand of $T$, when interpreted as a function on $\indec{\clustcat{Q}}$.
\end{thm}
\begin{proof}
While this result is well-known, see e.g.\ \cite[Rem.~5.6]{baurfriezes}, we are not able to find a convenient reference or direct proof in the literature. However, the result can be obtained by `direct calculation', using earlier results classifying or counting frieze patterns, for example those of Conway and Coxeter \cite{conwaytriangulated1,conwaytriangulated2}.
\end{proof}

\begin{rem}
Given any Dynkin quiver $Q$, not necessarily of type $\mathsf{A}_n$, one can extend the definition of frieze pattern to obtain functions $\indec{\clustcat{Q}}\to\ZZ_{>0}$ satisfying a condition generalising the $\SL_2$ diamond rule. However, in this larger generality, the analogue of Theorem~\ref{t:class2} is not true---there exist frieze patterns in this sense which do not take the value $1$ on all indecomposable summands of any cluster-tilting object. An example, and a discussion of these frieze patterns for $Q$ of type $\mathsf{D}_n$, can be found in a paper of Fontaine and Plamondon \cite{fontainecounting}.
\end{rem}

Recall from Remark~\ref{r:class1} that height $n$ friezes, thought of as functions on $\mathcal{D}_Q$ for $Q$ of type $\mathsf{A}_n$, are classified up to the glide reflection $\Sigma^{-1}\circ\tau$ by triangulations of labelled polygons. Comparing to Theorem~\ref{t:class2}, we see that there is a bijection between triangulations of a labelled $(n+3)$-sided polygon and the cluster-tilting objects of $\clustcat{Q}$ for $Q$ of type $\mathsf{A}_n$---indeed, this connection was used by Caldero, Chapoton and Schiffler \cite{calderoquivers1} to give a geometric description of $\clustcat{Q}$ in these cases. This bijection can be made explicit, and there is a combinatorial way of computing the endomorphism algebra of a cluster-tilting object $T\in\clustcat{Q}$ from the corresponding triangulation. These endomorphism algebras are called \emph{cluster-tilted algebras}, and have many interesting properties \cite{BIRSm}, which are discussed further in Schiffler's book \cite{schifflerquiver}.

\begin{eg}
The cluster category of type $\mathsf{A}_6$ has a cluster-tilting object whose indecomposable summands are indicated by the white vertices in the following figure.
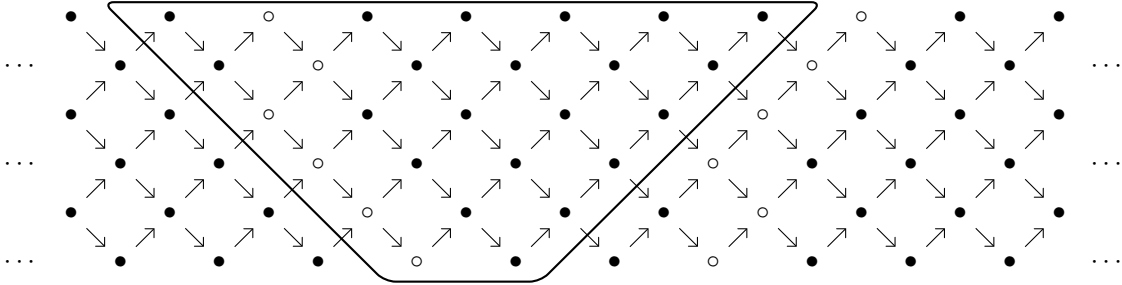
\begin{center}
\resizebox{0.95\textwidth}{!}{\begin{tikzpicture}[scale=0.75]
\foreach \x in {1,2,3,5,6,7,9,10,11}
      \draw (2*\x,2) node ({\x}c{2}) {$\bullet$};
\foreach \x in {1,2,4,5,6,7,9,10,11}
      \draw (2*\x,4) node ({\x}c{4}) {$\bullet$};
\foreach \x in {1,2,4,5,6,7,8,10,11}
      \draw (2*\x,6) node ({\x}c{6}) {$\bullet$};
\foreach \x in {1,2,3,5,6,8,9,10}
      \draw (2*\x+1,1) node ({\x}c{1}) {$\bullet$};
\foreach \x in {1,2,4,5,6,8,9,10}
      \draw (2*\x+1,3) node ({\x}c{3}) {$\bullet$};
\foreach \x in {1,2,4,5,6,7,9,10}
      \draw (2*\x+1,5) node ({\x}c{5}) {$\bullet$};
\foreach \x/\y in {4/1,3/3,3/5,8/5,7/3,7/1}
   \draw (2*\x+1,\y) node ({\x}c{\y}) {$\circ$};
\foreach \x/\y in {4/2,3/4,3/6,9/6,8/4,8/2}
   \draw (2*\x,\y) node ({\x}c{\y}) {$\circ$};
\foreach \x in {0,11}
   \foreach \y in {1,3,5}
      \draw (2*\x+1,\y) node {$\cdots$};
\foreach \x in {1,2,3,4,5,6,7,8,9,10}
   \foreach \y/\z in {2/1,4/3,6/5}
      \draw[->] ({\x}c{\y})--({\x}c{\z});
\foreach \x in {1,2,3,4,5,6,7,8,9,10}
   \foreach \y/\z in {2/3,4/5}
      \draw[->] ({\x}c{\y})--({\x}c{\z});
\foreach \x/\w in {1/2,2/3,3/4,4/5,5/6,6/7,7/8,8/9,9/10,10/11}
   \foreach \y/\z in {1/2,3/4,5/6}
      \draw[->] ({\x}c{\y})--({\w}c{\z});
\foreach \x/\w in {1/2,2/3,3/4,4/5,5/6,6/7,7/8,8/9,9/10,10/11}
   \foreach \y/\z in {3/2,5/4}
      \draw[->] ({\x}c{\y})--({\w}c{\z});
\draw[rounded corners,thick]
(2.65,6.3) -- (17.2,6.3) -- (11.5,0.6) -- (8.35,0.6) -- cycle;
\end{tikzpicture}}
\end{center}
This cluster-tilting object corresponds to our first example of a frieze pattern---observe that it has the same `shape' as the lightning bolt used to construct that frieze.

The cluster category of type $\mathsf{A}_3$ has a cluster-tilting object indicated by white vertices as shown.
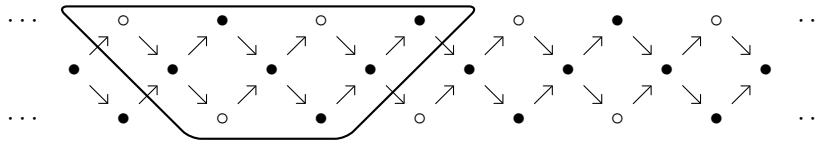
\[
\begin{tikzpicture}[scale=0.8]
\foreach \x in {1,3,5,7}
      \draw (2*\x+1,1) node ({\x}c{1}) {$\bullet$};
\foreach \x in {1,2,3,4,5,6,7,8}
      \draw (2*\x,2) node ({\x}c{2}) {$\bullet$};
\foreach \x in {2,4,6}
      \draw (2*\x+1,3) node ({\x}c{3}) {$\bullet$};
\foreach \x in {2,4,6}
   \draw (2*\x+1,1) node ({\x}c{1}) {$\circ$};
\foreach \x in {1,3,5,7}
   \draw (2*\x+1,3) node ({\x}c{3}) {$\circ$};
\foreach \x in {0,8}
   \foreach \y in {1,3}
      \draw (2*\x+1,\y) node {$\cdots$};
\foreach \x in {1,2,3,4,5,6,7}
      \draw[->] ({\x}c{2})--({\x}c{1});
\foreach \x in {1,2,3,4,5,6,7}
      \draw[->] ({\x}c{2})--({\x}c{3});
\foreach \x/\y in {1/2,2/3,3/4,4/5,5/6,6/7,7/8}
      \draw[->] ({\x}c{1})--({\y}c{2});
\foreach \x/\y in {1/2,2/3,3/4,4/5,5/6,6/7,7/8}
      \draw[->] ({\x}c{3})--({\y}c{2});
\draw[rounded corners,thick]
(1.65,3.3) -- (10.2,3.3) -- (7.5,0.6) -- (4.35,0.6) -- cycle;
\end{tikzpicture}
\]
This cluster-tilting object corresponds to the frieze pattern with no lightning bolt observed in Example~\ref{eg:no-bolt}. On the level of the cluster algebra, it corresponds to a seed whose quiver has an oriented cycle (and indeed consists entirely of such a cycle, of length $3$).
\end{eg}

Just as clusters can be mutated, replacing a single cluster variable by a new one, so can cluster-tilting objects, by an operation replacing a single indecomposable summand by a non-isomorphic one. Given what we have already seen, this operation could simply be defined using the bijections of Theorem~\ref{t:clust-bij}. However, mutation of cluster-tilting objects can also be defined intrinsically \cite{BMRRT}---indeed Iyama and Yoshino \cite{IyaYos} show that this is a general phenomenon of $2$-Calabi--Yau triangulated categories---and one can then check that the bijection between clusters and cluster-tilting objects from Theorem~\ref{t:clust-bij} relates the combinatorial and categorical notions of mutation. Under the bijection with triangulations of the $(n+3)$-gon, when $Q$ has type $\mathsf{A}_n$, mutations correspond to flips: each diagonal separates two triangles whose union is a quadrilateral, and using the other diagonal in this quadrilateral gives a new triangulation.
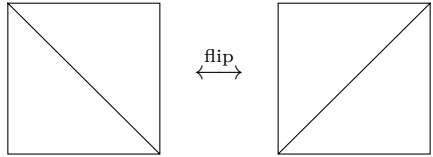
\[
\mathord{\begin{tikzpicture}[baseline=0]
\draw (-1,1) -- (1,1) -- (1,-1) -- (-1,-1) -- (-1,1) -- (1,-1);
\end{tikzpicture}}
\quad\stackrel{\text{flip}}{\longleftrightarrow}\quad
\mathord{\begin{tikzpicture}[xscale=-1,baseline=0]
\draw (-1,1) -- (1,1) -- (1,-1) -- (-1,-1) -- (-1,1) -- (1,-1);
\end{tikzpicture}}
\]

The connection between cluster algebras or categories and triangulations can be greatly generalised \cite{FST1}. Let $S$ be an oriented surface with boundary, and let $\mathbb{M}\subset\partial S$ be a finite subset of points in the boundary of $S$, such that each boundary component contains at least one point in $\mathbb{M}$. A triangulation of $(S,\mathbb{M})$ consists of a maximal collection of pairwise non-crossing arcs in $S$ with endpoints in $\mathbb{M}$. Excluding a small number of degenerate cases, any such triangulation determines a cluster-tilting object in a (generalised \cite{Amiot-ClustCat}) cluster category, and mutating this cluster-tilting object corresponds to flipping diagonals in the triangulation. The combinatorial rule for computing endomorphism algebras extends to this generality as well \cite{labardiniquivers}.

\appendix
\section{The bounded derived category}
\label{s:bdcat}
In this appendix, we will give a more common and more general description of the bounded derived category than that given in Definition~\ref{d:bdcat} in the case of quiver representations.

To start with, we let $\mathcal{A}$ be any abelian category---our main example is the abelian category $\rep{Q}$, but other examples include categories of representations over more general rings or algebras, or the category of abelian groups. We can then consider complexes of objects of $\mathcal{A}$, i.e.\ diagrams
\[\begin{tikzcd}V^\bullet\colon\cdots\arrow{r}&V^{i-1}\arrow{r}{d^{i-1}}&V^i\arrow{r}{d^i}&V^{i+1}\arrow{r}&\cdots\end{tikzcd}\]
consisting of objects $V^i\in\mathcal{A}$ for each $i\in\ZZ$ and morphisms $d^i\colon V^i\to V^{i+1}$ between these objects, with the property that $d^{i}\circ d^{i-1}=0$ for all $i$. A morphism of complexes is, similar to a morphism of quiver representations, a commutative diagram
\[\begin{tikzcd}
V^\bullet\arrow[shift right=11pt]{d}[swap]{\varphi}\colon\cdots\arrow{r}&V^{i-1}\arrow{r}{d^{i-1}}\arrow{d}[swap]{\varphi^{i-1}}&V^i\arrow{r}{d^i}\arrow{d}[swap]{\varphi^i}&V^{i+1}\arrow{r}\arrow{d}[swap]{\varphi^{i+1}}&\cdots\\
W^\bullet\colon\cdots\arrow{r}&W^{i-1}\arrow{r}[swap]{\delta^{i-1}}&W^i\arrow{r}[swap]{\delta^i}\arrow{r}&W^{i+1}\arrow{r}&\cdots
\end{tikzcd}\]

Given a complex $V^\bullet$ of objects of $\mathcal{A}$, we can compute its cohomology groups
\[\cohom{i}{V^\bullet}=\ker(d^i)/\im(d^{i-1}),\]
which are themselves objects of $\mathcal{A}$, and a morphism $\varphi\colon V^\bullet\to W^\bullet$ induces a morphism $\bar{\varphi}^i\colon\cohom{i}{V^{\bullet}}\to\cohom{i}{W^{\bullet}}$ in $\mathcal{A}$ for each $i$. We call $\varphi$ a \emph{quasi-isomorphism} if all of the maps $\bar{\varphi}^i$ are isomorphisms; note that this does \emph{not} imply the existence of a map $\psi\colon W^\bullet\to V^\bullet$ of complexes such that $\bar{\psi}^i=(\bar{\varphi}^i)^{-1}$.

\begin{defn}
Let $\mathcal{A}$ be an abelian category. The \emph{bounded derived category} $\bdcat{\mathcal{A}}$ of $\mathcal{A}$ has as objects complexes $V^\bullet$ of objects in $\mathcal{A}$ such that $\cohom{i}{V^\bullet}=0$ for $i\gg0$ and $i\ll0$. The morphism sets in $\bdcat{\mathcal{A}}$ are obtained from morphisms of complexes by formally adjoining inverses of all quasi-isomorphisms.\footnote{This language is rather sloppy, but will do for our purposes. The reader interested in more details should look up `localisation of categories', which is a similar construction to localisation of rings. In this language, we define $\bdcat{\mathcal{A}}$ by taking the category of bounded complexes over $\mathcal{A}$, which has the same objects as $\bdcat{\mathcal{A}}$ but morphisms given simply by morphisms of complexes, and then localising this category in the set of quasi-isomorphisms.} A little more concretely, this means that a morphism $V^\bullet\to W^\bullet$ in $\bdcat{\mathcal{A}}$ can be represented (non-uniquely) by a finite sequence
\[\begin{tikzcd}[row sep=small, column sep=small]
&X_1^\bullet\arrow{dl}[swap]{\sim}\arrow{dr}&&X_2^\bullet\arrow{dl}[swap]{\sim}\arrow{dr}&&\cdots\arrow{dl}[swap]{\sim}\arrow{dr}&&X_n^\bullet\arrow{dl}[swap]{\sim}\arrow{dr}\\
V^\bullet&&Y_1^\bullet&&Y_2^\bullet&&Y_{n-1}^\bullet&&W^\bullet
\end{tikzcd}\]
in which the rightward-pointing arrows are arbitrary maps of complexes, and the leftward-pointing arrows are quasi-isomorphisms.
\end{defn}

As mentioned in Section~\ref{s:clust-cat} for the special case $\mathcal{A}=\rep{Q}$, we may equip $\bdcat{\mathcal{A}}$ with the structure of a triangulated category. Part of this structure is the autoequivalence $\Sigma\colon\bdcat{\mathcal{A}}\isoto\bdcat{\mathcal{A}}$, which shifts the degrees of the objects and morphisms in a complex in the following way.
\[\begin{tikzcd}[row sep=5pt]
V^\bullet\colon\cdots\arrow{r}&V^{i-1}\arrow{r}{d^{i-1}}&V^i\arrow{r}{d^i}&V^{i+1}\arrow{r}&\cdots\\
\Sigma V^\bullet\colon\cdots\arrow{r}&V^i\arrow{r}{-d^i}&V^{i+1}\arrow{r}{-d^{i+1}}&V^{i+2}\arrow{r}&\cdots
\end{tikzcd}\]
The category $\bdcat{\mathcal{A}}$ also has direct sums, given by taking the direct sum of complexes term by term.

For a general abelian category $\mathcal{A}$, understanding $\bdcat{\mathcal{A}}$ can be rather difficult---indeed, from the description given above it is not clear that it is a category at all, for the rather frightening reason that the morphisms between a pair of objects, which consist of sequences of morphisms of complexes interlaced with formal inverses of quasi-isomorphisms, need not a priori form a well-defined set, although it turns out that they do. In some special cases (most notably the case $\mathcal{A}=\rep{Q})$, simpler descriptions are available. First, we give a general simplification of the description of the morphisms in $\bdcat{\mathcal{A}}$.

\begin{defn}
We say a morphism $f\colon V^\bullet\to W^\bullet$ of complexes is \emph{null-homotopic} if there are morphisms $h^i\colon V^i\to W^{i-1}$ for each $i\in\ZZ$, such that $\varphi^i=\delta^{i-1}h^i+h^{i+1}d^i$, as in the following figure.
\[\begin{tikzcd}[row sep=large, column sep=large]
V^\bullet\arrow[shift right=11pt]{d}[swap]{\varphi}\colon\cdots\arrow{r}&V^{i-1}\arrow{r}{d^{i-1}}\arrow{d}[swap]{\varphi^{i-1}}&V^i\arrow{r}{d^i}\arrow{d}[swap]{\varphi^i}\arrow{dl}[swap]{h^i}&V^{i+1}\arrow{r}\arrow{d}[swap]{\varphi^{i+1}}\arrow{dl}[swap]{h^{i+1}}&\cdots\\
W^\bullet\colon\cdots\arrow{r}&W^{i-1}\arrow{r}[swap]{\delta^{i-1}}&W^i\arrow{r}[swap]{\delta^i}\arrow{r}&W^{i+1}\arrow{r}&\cdots
\end{tikzcd}\]
Let $\bhcat{\mathcal{A}}$ denote the category with the same objects as $\bdcat{\mathcal{A}}$, and with the morphism space $\Hom_{\bhcat{\mathcal{A}}}(V^\bullet,W^\bullet)$ given by the space of maps of complexes $V^\bullet\to W^\bullet$, modulo the subspace of null-homotopic maps.
\end{defn}

\begin{prop}
\label{p:homotopy}
The bounded derived category $\bdcat{\mathcal{A}}$ is equivalent to the category obtained from $\bhcat{\mathcal{A}}$ by formally inverting (the classes of) quasi-isomorphisms.
\end{prop}

The advantage of describing $\bdcat{\mathcal{A}}$ as in Proposition~\ref{p:homotopy} is that it allows for an easier description of the morphisms; in this guise, morphisms $V^\bullet\to W^\bullet$ are represented (still non-uniquely) by \emph{roofs} of the form
\[\begin{tikzcd}[row sep=small,column sep=small]
&X^\bullet\arrow{dl}[swap]{\sim}\arrow{dr}\\
V^\bullet&&W^\bullet
\end{tikzcd}\]
where both arrows represent morphisms in $\bhcat{\mathcal{A}}$ (that is, homotopy classes of maps of complexes), with the left-hand arrow being a quasi-isomorphism. This advantage derives from the fact that $\bhcat{\mathcal{A}}$, unlike the category with the same objects but with all morphisms of complexes, is a triangulated category, and the localisation in quasi-isomorphisms can be realised as a Verdier localisation \cite{Verdier-These} in the triangulated subcategory of acyclic complexes, those complexes $V^\bullet$ for which $\cohom{i}{V^\bullet}=0$ for all $i$.

Now let us consider the case that $\mathcal{A}=\fgmod{A}$ is the category of finite-dimensional modules over a finite-dimensional algebra $A$; we abbreviate $\bdcat{A}\defeq\bdcat{\fgmod{A}}$ and $\bhcat{A}\defeq\bhcat{\fgmod{A}}$. (Much weaker, but more technical, assumptions are sufficient for what follows.) Recall that an $A$-module $P$ is projective if for every surjective $A$-module map $f\colon X\to Y$ and every $A$-module map $g\colon P\to Y$, there exists $h\colon P\to X$ with $g=f\circ h$.
\[\begin{tikzcd}
&P\arrow{d}{g}\arrow[dashed]{dl}[swap]{\exists h}\\
X\arrow{r}{f}&Y\arrow{r}&0
\end{tikzcd}\]

\begin{defn}
\label{d:kbproj}
Write $\bhcat{\proj{A}}$ for the category whose objects are complexes $P^\bullet$ of $A$-modules such that $P^i$ is projective for all $i$, and $P^i=0$ for $i\ll0$ or $i\gg0$, and with morphisms given by maps of complexes modulo null-homotopic maps. Using the description of $\bdcat{A}$ from Proposition~\ref{p:homotopy}, we see that there is a natural functor $\bhcat{\proj{A}}\to\bdcat{A}$ given by the identity on both objects and morphisms.
\end{defn}

Note that the boundedness conditions in Definition~\ref{d:kbproj} are different from those in Definition~\ref{d:bdcat}, since they refer to the terms of the complex, not the cohomology. Indeed, since $\proj{A}$ is typically not an abelian category, an object $P^\bullet\in\bhcat{\proj{A}}$ does not have well-defined cohomologies $\cohom{i}{P^\bullet}\in\proj{A}$---they are well-defined objects of $\fgmod{A}$, but are not typically projective.

\begin{defn}
We say that a finite-dimensional algebra $A$ has \emph{finite global dimension} if there is an integer $n$ such that each $X\in\fgmod{A}$ fits into an exact sequence
\begin{equation}
\label{eq:projres}
\begin{tikzcd}
0\arrow{r}&P^n\arrow{r}&\cdots\arrow{r}&P^0\arrow{r}{\pi}&X\arrow{r}&0
\end{tikzcd}
\end{equation}
in which each $P^i$ is a projective $A$-module. The minimal such $n$ is called the \emph{global dimension} of $A$, which is defined to be $\infty$ if no such $n$ exists.
\end{defn}
Exactness of the complex \eqref{eq:projres} is equivalent to the morphism
\[\begin{tikzcd}
\cdots\arrow{r}&0\arrow{r}&P^n\arrow{r}&\cdots\arrow{r}&P^1\arrow{r}&P^0\arrow{d}{\pi}\arrow{r}&0\arrow{r}&\cdots\\
\cdots\arrow{r}&0\arrow{r}&0\arrow{r}&\cdots\arrow{r}&0\arrow{r}&X\arrow{r}&0\arrow{r}&\cdots
\end{tikzcd}\]
in $\bhcat{A}$ being a quasi-isomorphism, so that the lower complex is isomorphic in $\bdcat{A}$ to a complex in the image of the natural map $\bhcat{\proj{A}}\to\bdcat{A}$. This observation can be upgraded to the following theorem.

\begin{thm}[{\cite[Ch.~II, Prop.~1.4]{verdiercategories1}}]
\label{t:fin-gldim}
If $A$ is a finite-dimensional algebra with finite global dimension, the natural map $\bhcat{\proj{A}}\to\bdcat{A}$ is an equivalence.
\end{thm}

Theorem~\ref{t:fin-gldim} is very useful, since the category $\bhcat{\proj{A}}$ has a much simpler description than the (equivalent) category $\bdcat{A}$. This theorem covers the case of derived categories of representations of acyclic quivers, which was of interest to us in Section~\ref{s:clust-cat}. Such a quiver $Q$ has a path algebra $\KK Q$, which as a vector space is spanned by the paths of $Q$---this is finite dimensional since $Q$ is acyclic. Multiplication of paths is given by concatenation when this is defined, and $0$ when it is not (cf.\ Remark~\ref{r:pathcat}). Thus $\KK Q$ is a finite-dimensional algebra, and $\rep{Q}$ is equivalent (even isomorphic) to the category $\fgmod{\KK Q}$ of finite-dimensional modules over this algebra. The global dimension of $\KK Q$ is $1$, and so there is an equivalence $\bhcat{\proj{\KK Q}}\isoto\bdcat{\KK Q}$ by Theorem~\ref{t:fin-gldim}.

An algebra $A$ of global dimension $1$ is called \emph{hereditary} (because the property of being a projective $A$-module is inherited by submodules---that is, any submodule of a projective $A$-module is again projective). In this case we can simplify our description of $\bdcat{A}$ still further, to recover our ad hoc description of the objects of $\bdcat{Q}$ from Section~\ref{s:clust-cat}.

\begin{prop}[{\cite[Lem.~4.1]{happelderived}}]
Let $A$ be a hereditary algebra. Then each object $V^\bullet\in\bdcat{A}$ is isomorphic to
\[\begin{tikzcd}
\cdots\arrow{r}{0}&\cohom{i-1}{V^\bullet}\arrow{r}{0}&\cohom{i}{V^\bullet}\arrow{r}{0}&\cohom{i+1}{V^{\bullet}}\arrow{r}{0}&\cdots\end{tikzcd}\]
In other words, if $\Sigma^{-i}M$ denotes the \emph{stalk complex} with $M\in\fgmod{A}$ in degree $i$ and the zero module in all other degrees, then we have
\[V^\bullet\cong\bigoplus_{i\in\ZZ}\Sigma^{-i}\cohom{i}{V^\bullet}.\]
\end{prop}

Our description of morphisms between stalk complexes in $\bdcat{Q}$ from Section~\ref{s:clust-cat} can also be recovered from the results presented in this appendix, but we will not give the details of this. Roughly, the idea is to use Theorem~\ref{t:fin-gldim} to replace each stalk complex by a quasi-isomorphic complex of projectives, and then compare the computation of the morphisms between these complexes up to homotopy to the usual computation of extension groups between the two representations giving the non-zero terms of the stalk complexes.

\section*{Acknowledgements}
These notes were written to accompany a series of three lectures given in October 2020 at the LMS Autumn Algebra School, funded by the London Mathematical Society, the European Research Council and the International Centre for Mathematical Sciences in Edinburgh. I would like to thank David Jordan, Nadia Mazza and Sibylle Schroll for organising the lecture series, Karin Baur and Bethany Marsh for their comments on early drafts of the notes, and the students attending the lectures for their interest and insightful questions.

\defbibheading{bibliography}[\refname]{\section*{#1}}\printbibliography
\end{document}